\documentclass[preprint,11pt]{elsarticle}%
\usepackage{amssymb}
\usepackage{amsmath,amsfonts}
\usepackage{subfigure,epstopdf}
\usepackage{amsmath}
\usepackage{amsfonts}
\usepackage{graphicx}%
\providecommand{\U}[1]{\protect\rule{.1in}{.1in}}
\newtheorem{theorem}{Theorem}
\newtheorem{lemma}[theorem]{Lemma}
\newdefinition{remark}{Remark}
\newproof{proof}{Proof}
\newtheorem{proposition}[theorem]{Proposition}
\newtheorem{corollary}[theorem]{Corollary}
\begin{document}
\begin{frontmatter}
\title{On  the  mean field approximation of a stochastic model of tumor-induced angiogenesis}
\author{Vincenzo Capasso \corref{cor1}}
\ead{vincenzo.capasso@unimi.it}
\address{ADAMSS (Centre for Advanced Applied Mathematical and Statistical Sciences), \\  Universit\'a degli Studi di Milano, \\ Via Saldini 50. 20133 Milano, Italy}
\author{Franco  Flandoli \corref{}}
\ead{flandoli@dma.unipi.it}
\address{Department of Mathematics, University of  Pisa,\\
Largo Pontecorvo 5, 56127 Pisa, Italy }
\cortext[cor1]{Corresponding author}
\begin{abstract}
In the field of Life Sciences it it  very common to deal with
extremely complex systems, from both  analytical and computational
points of view, due to the unavoidable coupling of different
interacting structures.  As an example, angiogenesis has revealed
to be an highly complex, and  extremely interesting biomedical
problem, due  to the strong coupling between the kinetic
parameters   of the relevant  branching -  growth - anastomosis
stochastic processes of the capillary network, at the microscale,
and   the family of interacting  underlying biochemical fields, at
the macroscale. In this paper  an original revisited    conceptual
stochastic model of tumor  driven angiogenesis  has been proposed,
for which it has been shown that it is possible to reduce
complexity by taking advantage of the intrinsic multiscale
structure of the system; one may  keep the stochasticity of the
dynamics of the vessel tips at their natural microscale, whereas
the dynamics of the underlying fields is given by a deterministic
mean field approximation obtained by an averaging at a suitable
mesoscale. While in previous papers only an heuristic
justification of this approach had been offered, in this paper a
rigorous proof is given of the so called ``propagation of chaos",
which leads to a mean field approximation of the  stochastic
relevant measures associated with the vessel dynamics, and
consequently of the underlying TAF field. As  a side though
important  result, the non-extinction of the random process  of
tips has been  proven during any finite time interval.
\end{abstract}
\end{frontmatter}

\section{Introduction}

In Life Sciences we may observe a wide spectrum of self-organization
phenomena. In most of these phenomena, randomness plays a major role; see
\cite{VK_2014} for a general discussion. As a working example, in this paper
we refer to tumor-driven angiogenesis; in this case cells organize themselves
as a capillary network of vessels, the organization being driven by a family
of underlying fields, such as nutrients, growth factors and alike
\cite{folkman_1974,jain-carmeliet,carmeliet_2005}. Indeed an angiogenic system
is extremely complex due to its intrinsic multiscale structure. We need to
consider the strong coupling between the kinetic parameters of the relevant
stochastic processes describing branching, vessel extension, and anastomosis
of the capillary network at the microscale, and the family of interacting
underlying fields at the macroscale \cite{chaplain_1998, Hubbard_etal_2009,
othmer, VK_morale_jomb, cotter}.

The kinetic parameters of the mentioned stochastic processes depend on the
concentrations of certain chemical factors which satisfy reaction-diffusion
equations (RDEs) \cite{chaplain_1998,harrington,tong}. Viceversa, the RDEs for
such underlying fields contain terms that depend on the spatial distribution
of vascular cells. As a consequence, a full mathematical model of angiogenesis
consists of the (stochastic) evolution of vessel cells, coupled with a system
of RDEs containing terms that depend on the distribution of vessels. The
latter is random and therefore the equations for the underlying fields are
random RDEs, thus inducing randomness in the kinetic parameters of the
relevant stochastic geometric processes describing the evolution of the vessel
network; we might say that the vessel dynamics is a \textquotedblleft
doubly\textquotedblright\ stochastic process.

This strong coupling leads to an highly complex mathematical problem from both
analytical and computational points of view. A possibility to reduce
complexity is offered by the so called hybrid models, which exploit the
natural multiscale nature of the system.

The idea consists of approximating the random RDEs by deterministic ones, in
which the microscale (random) terms depending on cell distributions are
replaced by their (deterministic) mesoscale averages. In this way the
mentioned kinetic parameters may be taken as depending on the mean field
approximation of the underlying fields, thus leading to a \textquotedblleft
simple" stochasticity of the random processes of branching - vessel extension
- anastomosis \cite{BCACT}.

In the literature there are examples of rigorous derivations of mean field
equations of stochastic particle dynamics \cite{oel2,sznitman,C_M},
\cite{Meleard}; however, to the best of the authors' knowledge, for the kind
of models considered here, a rigorous proof of the required \textquotedblleft
propagation of chaos" has not yet been given, though previous attempts have
led to heuristic derivations (see \cite{BCAC3}, \cite{BCACT} and references therein).

Eventually, in this paper the authors have been able to derive mean field
equations with the required, non trivial, rigorous approach. As a side result
to understand the impact of anastomosis, in the Appendix, it has been proven
that the random measure of tips never vanishes during any finite time interval
(see \ref{B}).

The proof that the number of new tips cannot growth without
control, given in Section \ref{section upper bound}, is highly
nontrivial. This is the first work that deals rigorously with this
question, namely the size of growth when tips may emerge from
created vessels and the length of the vessel is potentially
unbounded, in finite time, due to the Gaussian fluctuations of the
noise. The usual control from above by a Yule process does not
work here and new tools have been used. At the technical level,
let us also highlight the proof of uniqueness of measure-valued
solutions, that seems to be original with respect to the related
literature.

The paper is organized as follows. Section \ref{sec:ANGIO_model} introduces
our stochastic model and all relevant random measures associated with. Section
\ref{sec:evolution_empirical} is devoted to the evolution of the empirical
measure and an heuristic derivation of the mean field equations for the
deterministic measure of tips, and the associated TAF concentrations, based on
a conjectured \textquotedblleft propagation of chaos\textquotedblright.
Section \ref{sec:main_results} presents our main mathematical results. Section
\ref{sct proof thm convergence} contains a detailed proof of Theorem
\ref{thm convergence}, as far as the tightness of the sequence of laws of
$(Q_{N},C_{N})_{N\in\mathbb{N}}$ is concerned, and consequently the existence
of a weakly convergence subsequence, thus anticipating the existence part
claimed in Theorem \ref{thm uniqueness}. All required estimates are rigorously
derived here. Section \ref{sct uniqueness} is devoted to the proof of Theorem
\ref{thm uniqueness}, as far as the claimed uniqueness is concerned.

\section{A mathematical model for tumor induced
angiogenesis\label{sec:ANGIO_model}}

The main features of the process of formation of a tumor-driven vessel network
are (see \cite{chaplain_stuart:93}, \cite{Hubbard_etal_2009}, \cite{BCACT})

\begin{itemize}
\item[i)] vessel branching;

\item[ii)] vessel extension;

\item[iii)] chemotaxis in response to a generic tumor \ angiogenic factor
(TAF), released by tumor cells;

\item[iv)] haptotactic migration in response to fibronectin gradient, emerging
from the extracellular matrix and through degradation and production by
endothelial cells themselves;

\item[v)] anastomosis, \ the coalescence of a capillary tip with an existing vessel.
\end{itemize}

We will limit ourselves to describe the dynamics of tip cells at the front of
growing vessels, as a consequence of chemotaxis in response to a generic tumor
factor (TAF) released by tumor cells, in a space $\mathbb{R}^{d},$ of
dimension $d\in\{2,3\}$.

The number of tip cells changes in time, due to proliferation and death. We
shall denote by $N_{t}$ the random number of active tip cells at time
$t\in\mathbb{R}_{+}.$ We shall refer to $N:=N_{0}$ as the scale parameter of
the system. The $i$-th tip cell is characterized by the random variables
$T^{i,N}$ and $\Theta^{i,N},$ representing the birth (branching) and death
(anastomosis) times, respectively, and by its position and velocity $\left(
\mathbf{X}^{i,N}\left(  t\right)  ,\mathbf{V}^{i,N}\left(  t\right)  \right)
\in\mathbb{R}^{2d}$, $t\in\lbrack T^{i,N},\Theta^{i,N}).$ Its entire history
is then given by the stochastic process
\[
\left(  \mathbf{X}^{i,N}\left(  t\right)  ,\mathbf{V}^{i,N}\left(  t\right)
\right)  _{t\in\lbrack T^{i,N},\Theta^{i,N})}.
\]

All random variables and processes are defined on a filtered probability space
$\left(  \Omega,\mathcal{F},\mathcal{F}_{t},\mathbb{P}\right)  $.

The growth factor is a random function $C_{N}:\Omega\times\lbrack
0,\infty)\times\mathbb{R}^{d}\rightarrow\mathbb{R}$, that we write as
$C_{N}\left(  t,\mathbf{x}\right)  $.

Tip cells and growth factor satisfy the stochastic system%
\begin{align}
d\mathbf{X}^{i,N}\left(  t\right)   &  =\mathbf{V}^{i,N}\left(  t\right)
dt\label{vessel_extension}\\
d\mathbf{V}^{i,N}\left(  t\right)   &  =\left[  -k_{1}\mathbf{V}^{i,N}\left(
t\right)  +f\left(  |\nabla C_{N}\left(  t,\mathbf{X}^{i,N}\left(  t\right)
\right)  |\right)  \nabla C_{N}\left(  t,\mathbf{X}^{i,N}\left(  t\right)
\right)  \right]  dt\nonumber\\
&  +\sigma d\mathbf{W}^{i}\left(  t\right)  \label{velocity}%
\end{align}%
\begin{equation}
\partial_{t}C_{N}\left(  t,\mathbf{x}\right)  =k_{2}\delta_{A}\left(
\mathbf{x}\right)  +d_{1}\Delta C_{N}\left(  t,\mathbf{x}\right)  -\eta\left(
t,\mathbf{x},\left\{  Q_{N}\left(  s\right)  \right\}  _{s\in\left[
0,t\right]  }\right)  C_{N}\left(  t,\mathbf{x}\right)  \label{eq CN}%
\end{equation}
where $k_{1},k_{2},\sigma,d_{1}>0$, are given; $\mathbf{W}^{i}\left(
t\right)  $, $i\in N$, are independent Brownian motions; $A$ is a Borel set of
$\mathbb{R}^{d},$ representing the tumoral region acting as a source of TAF;
$\delta_{A}$ is the generalized derivative of the Dirac measure $\epsilon_{A}$
with respect to the usual Lebesgue measure $L^{d}$ on $B_{\mathbb{R}^{d}},$
i.e. for any $B\in B_{\mathbb{R}^{d}},$
\[
\epsilon_{A}(B)=\int_{B}\delta_{A}(\mathbf{x})d\mathbf{x}=\mathcal{L}%
^{d}(B\cap A).
\]

The initial condition $C_{N}\left(  0,\mathbf{x}\right)  $ is also given,
while the initial conditions on $X^{i,N}\left(  t\right)  ,$ and
$V^{i,N}\left(  t\right)  $ depend upon the process at the time of birth
$T^{i,N}$ of the $i-$th tip;\ the term $\eta\left(  t,\mathbf{x},\left\{
Q_{N}\left(  s\right)  \right\}  _{s\in\left[  0,t\right]  }\right)  $ will be
described below.

In the Equation (\ref{velocity}), besides the friction force, there is a
chemotactic force due to the underlying TAF field $C_{N}(t,x);$ different from
relevant literature (see e.g. \cite{chaplain_1998}, \cite{plank_sleeman:03}),
here we assume that $f$ depends upon the absolute value of the gradient of the
TAF field
\[
f(|\nabla C_{N}(t,\mathbf{x})|)=\frac{d_{2}}{(1+\gamma_{1}|\nabla
C_{N}(t,\mathbf{x})|)^{q}}.
\]
This choice, requested by mathematical issues as necessary bounds, leads to
upper bounds of the term%

\begin{equation}
f\left(  |\nabla C_{N}\left(  t,\mathbf{X}^{i,N}\left(  t\right)  \right)
|\right)  \nabla C_{N}\left(  t,\mathbf{X}^{i,N}\left(  t\right)  \right)
\label{saturation}%
\end{equation}
for large values of the gradient of the TAF field; indeed this makes the model
more realistic, since it bounds the effect of possible large values of this
gradient. For $q=1,$ we would have a saturating limit value for the term in
Equation (\ref{saturation}).

Let us describe the term $\eta_{N}\left(  t,\mathbf{x},\left\{  Q_{N}\left(
s\right)  \right\}  _{s\in\left[  0,t\right]  }\right)  $. For every $t\geq0$,
we introduce the scaled measure on $\mathbb{R}^{d}$
\begin{equation}
Q_{N}\left(  t\right)  :=\frac{1}{N}\sum_{i=1}^{N_{t}}\mathbb{I}_{t\in\lbrack
T^{i,N},\Theta^{i,N})}\epsilon_{\left(  \mathbf{X}^{i,N}\left(  t\right)
,\mathbf{V}^{i,N}\left(  t\right)  \right)  }, \label{empirical meas}%
\end{equation}
where $\epsilon$ denotes the usual Dirac measure, having the Dirac delta
$\delta$ as its generalized density with respect to the usual Lebesgue
measure. With these notations, and denoting by $\mathcal{M}_{+}\left(
\mathbb{R}^{d}\times\mathbb{R}^{d}\right)  $ the set of all finite positive
Borel measures on $\mathbb{R}^{d}\times\mathbb{R}^{d}$, we may assume that,
for every $t\geq0$, the function $\eta\left(  t,\cdot,\cdot\right)  $ maps
$\mathbb{R}^{d}\times L^{\infty}\left(  0,t;\mathcal{M}_{+}\left(
\mathbb{R}^{d}\times\mathbb{R}^{d}\right)  \right)  $ into $\mathbb{R}$:%
\[
\eta\left(  t,\cdot,\cdot\right)  :\mathbb{R}^{d}\times L^{\infty}\left(
0,t;\mathcal{M}_{+}\left(  \mathbb{R}^{d}\times\mathbb{R}^{d}\right)  \right)
\rightarrow\mathbb{R},
\]
for which we will assume the following structure:%
\[
\eta\left(  t,\mathbf{x},\left\{  Q_{N}\left(  s\right)  \right\}
_{s\in\left[  0,t\right]  }\right)  =\int_{0}^{t}\left(  \int_{\mathbb{R}%
^{d}\times\mathbb{R}^{d}}K_{1}\left(  \mathbf{x}-\mathbf{x}^{\prime}\right)
\left\vert \mathbf{v}^{\prime}\right\vert Q_{N}\left(  s\right)  \left(
d\mathbf{x}^{\prime},d\mathbf{v}^{\prime}\right)  \right)  ds
\]
for a suitable smooth bounded kernel $K_{1}:\mathbb{R}^{d}\rightarrow
\mathbb{R}$.

\subsection{The capillary network}

The capillary network of endothelial cells $\mathbf{X}^{N}(t)$ consists of the
union of all random trajectories representing the extension of individual
capillary tips from the (random) time of birth (branching) $T^{i,N},$ to the
(random) time of death (anastomosis) $\Theta^{i,N},$
\begin{equation}
\mathbf{X}{^{N}}(t)=\bigcup_{i=1}^{N_{t}}\{\mathbf{X}^{i,N}\left(  s\right)
|T^{i,N}\leq s\leq\min\{t,\Theta^{i,N}\}\}, \label{network}%
\end{equation}
giving rise to a stochastic network. Thanks to the choice of a Langevin model
for the vessels extension, we may assume that the trajectories are
sufficiently regular and have an integer Hausdorff dimension $1.$

Hence \cite{capasso_villa_2008} the random measure
\begin{equation}
A\in\mathcal{B}_{\mathbb{R}^{d}}\mapsto\mathcal{H}^{1}(\mathbf{X}{^{N}}(t)\cap
A)\in\mathbb{R}_{+}%
\end{equation}
may admit a random generalized density $\delta_{\mathbf{X}{^{N}}(t)}(x)$ with
respect to the usual Lebesgue measure on $\mathbb{R}^{d}$ such that, for any
$A\in B_{\mathbb{R}^{d}},$
\begin{equation}
\mathcal{H}^{1}(\mathbf{X}{^{N}}(t)\cap A)=\int_{A}\delta_{\mathbf{X}{^{N}%
}(t)}(\mathbf{x})d\mathbf{x}.
\end{equation}

By Theorem 11 in \cite{capasso_flandoli_lucio}, we may then state that%

\begin{equation} \label{stoch_network}
H^{1}\left(  \mathbf{X}{^{N}}(t)\cap A\right)  =\int_{0}^{t}\sum_{i=1}^{N_{s}%
}\epsilon_{\mathbf{X}^{i,N}\left(  s\right)  }\left(  A\right)  \left\vert
\frac{d}{ds}\mathbf{X}^{i,N}\left(  s\right)  \right\vert I_{s\in\lbrack
T^{i,N},\Theta^{i,N})}ds. %
\end{equation}

Hence
\begin{equation} \label{stoch_network_1}
\delta_{\mathbf{X}{^{N}}(t)}\left(  \mathbf{x}\right)
=\int_{0}^{t}\sum _{i=1}^{N_{s}}\delta_{\mathbf{X}^{i,N}\left(
s\right)  }\left( \mathbf{x}\right)  \left\vert
\frac{d}{ds}\mathbf{X}^{i,N}\left(  s\right) \right\vert
\mathbb{I}_{s\in\lbrack T^{i,N},\Theta^{i,N})}ds.
\end{equation}

With this is mind we may write%

\begin{align}
&  \eta\left(  t,\mathbf{x},\left\{  Q_{N}\left(  s\right)  \right\}
_{s\in\left[  0,t\right]  }\right)  =\nonumber\\
&  \qquad\quad=\frac{1}{N}\int_{0}^{t}ds\sum_{i=1}^{N_{s}}\mathbb{I}%
_{s\in\lbrack T^{i,N},\Theta^{i,N})}K_{1}(\mathbf{x}-\mathbf{X}^{i,N}\left(
s\right)  )|\mathbf{V}^{i,N}\left(  s\right)  |\nonumber\label{eta2}\\
&  \qquad\quad=\frac{1}{N}(K_{1}\ast\delta_{\mathbf{X}{^{N}}(t)})(\mathbf{x}).
\end{align}

\subsubsection{Branching}

Two kinds of branching have been identified, either from a tip or from a
vessel. The birth process of new tips can be described in terms of a marked
point process (see e.g. \cite{B}), by means of a random measure $\Phi$ on
$\mathcal{B}_{\mathbb{R}^{+}\times\mathbb{R}^{d}\times\mathbb{R}^{d}}$ such
that, for any $t\geq0$ and any $B\in\mathcal{B}_{\mathbb{R}^{d}\times
\mathbb{R}^{d}}$,
\begin{equation}
\Phi((0,t]\times B):=\int_{0}^{t}\!\int_{B}\Phi(ds\times d\mathbf{x}\times
d\mathbf{v}), \label{PHI}%
\end{equation}
where $\Phi(ds\times dx\times dv)$ is the random variable that counts those
tips born either from an existing tip, or from an existing vessel, during
times in $(s,s+ds]$, with positions in $(x,x+dx]$, and velocities in
$(v,v+dv]$.

As an additional simplification, we will further assume that the initial value
of the state of a new tip is $(\mathbf{X}_{T^{N_{t}+1,N}}^{N_{t}%
+1,N},\mathbf{V}_{T^{N_{t}+1,N}}^{N_{t}+1,N}),$ where $T^{N_{t}+1,N}$ is the
random time of branching, $\mathbf{X}_{T^{N_{t}+1,N}}^{N_{t}+1,N}$ is the
random point of branching, and $\mathbf{V}_{T^{N_{t}+1,N}}^{N_{t}+1,N}$ is a
random velocity, selected out of a probability distribution $G_{{v}_{0}}$ with
mean $v_{0}.$

Given the history $F_{t^{-}}$ of the whole process up to time $t^{-},$ we
assume that the compensator of the random measure $\Phi(ds\times dx\times dv)$
is given by
\begin{align*}
&  \alpha(C_{N}(s,\mathbf{x}))G_{{v}_{0}}(\mathbf{v})\,\sum_{i=1}^{N_{s}%
}\mathbb{I}_{s\in\lbrack T^{i,N},\Theta^{i,N})}\epsilon_{\mathbf{X}%
^{i,N}\left(  s\right)  }\left(  d\mathbf{x}\right)  \,d\mathbf{v}ds\\
&  +\beta(C_{N}(s,\mathbf{x}))G_{\mathbf{v}_{0}}(\mathbf{v})\epsilon
_{\mathbf{X}^{N}\left(  s\right)  }\left(  d\mathbf{x}\right)  \,d\mathbf{v}ds
\end{align*}
where $\alpha(C),\beta(C)$ are non-negative smooth functions, bounded with
bounded derivatives; for example, we may take
\[
\alpha(C)=\alpha_{1}\frac{C}{C_{R}+C},
\]
where $C_{R}$ is a reference density parameter \cite{VK_morale_jomb};\ and
similarly for $\beta(C)$.

The term corresponding to tip branching
\begin{equation}
\alpha(C_{N}(s,\mathbf{x}))G_{{v}_{0}}(\mathbf{v})\,\sum_{i=1}^{N_{s}%
}\mathbb{I}_{s\in\lbrack T^{i,N},\Theta^{i,N})}\epsilon_{\mathbf{X}%
^{i,N}\left(  s\right)  }\left(  d\mathbf{x}\right)  \,d\mathbf{v}ds
\label{tip_branching}%
\end{equation}
comes from the following argument:\ a new tip may arise only at positions
$\mathbf{X}^{i,N}\left(  s\right)  $ with $s\in\lbrack T^{i,N},\Theta^{i,N})$
(the positions of the tips existing at time $s$); the birth is modulated by
$\alpha(C_{N}(s,x))$, since we want to take into account the density of the
growth factor; and the velocity of the new tip is chosen at random with
density $G_{{v}_{0}}(v)$. It can be rewritten as%

\begin{equation}
N\alpha(C_{N}(s,x))G_{{v}_{0}}(v)dv\,\int_{\mathbb{R}^{d}}Q_{N}(s)\left(
d\mathbf{x},d\mathbf{v}\right)  \,ds, \label{branching_1}%
\end{equation}

The term corresponding to vessel branching%
\begin{equation}
\beta(C_{N}(s,\mathbf{x}))G_{\mathbf{v}_{0}}(\mathbf{v})\epsilon
_{\mathbf{X}^{N}\left(  s\right)  }\left(  d\mathbf{x}\right)  \,d\mathbf{v}ds
\label{vessel_branching}%
\end{equation}
tells us that a new tip may stem at time $s$ from a point $x$ belonging to the
stochastic network $X^{N}(s)$ already existing at time $s$, at a rate
depending on the concentration of the TAF via $\beta(C_{N}(s,x)),$ for the
reasons described above. Again the velocity of the new tip is chosen at random
with density $G_{{v}_{0}}(v)$.

Because of (\ref{stoch_network}) it can be rewritten as
\begin{equation}
N\beta(C_{N}(s,\mathbf{x}))G_{\mathbf{v}_{v_{0}}}(\mathbf{v})d\mathbf{v}%
\int_{0}^{s}\int_{\mathbb{R}^{d}}\left\vert \mathbf{v}\right\vert Q_{N}\left(
r\right)  \left(  d\mathbf{x},d\mathbf{v}\right)  drds \label{branching_2}%
\end{equation}

\subsubsection{Anastomosis}

When a vessel tip meets an existing vessel it joins it at that point and time
and it stops moving. This process is called tip-vessel anastomosis.

As in the case of the branching process, we may model this process via a
marked counting process; anastomosis is modelled as a \textquotedblleft
death\textquotedblright\ process.

Let $\Psi$ denote the random measure on $\mathcal{B}_{\mathbb{R}^{+}%
\times\mathbb{R}^{d}\times\mathbb{R}^{d}}$ such that, for any $t\geq0,$ and
any $B\in\mathcal{B}_{\mathbb{R}^{d}\times\mathbb{R}^{d}}$,
\begin{equation}
\Psi((0,t]\times B):=\int_{0}^{t}\int_{B}\Psi(ds\times d\mathbf{x}\times
d\mathbf{v}) \label{psiB}%
\end{equation}
where $\Psi(ds\times dx\times dv)$ is the random variable counting those tips
that are absorbed by the existing vessel network during time $(s,s+ds],$ with
position in $(x,x+dx],$ and velocity in $(v,v+dv]$.

We assume that the compensator of the random measure $\Psi(ds\times dx\times
dv)$ is%

\begin{align}
&  \gamma\sum_{i=1}^{N_{s}}\mathbb{I}_{s\in\lbrack T^{i,N},\Theta^{i,N}%
)}h\left(  \frac{1}{N}\left(  K_{2}\ast\delta_{\mathbf{X}^{N}\left(  s\right)
}\right)  \left(  \mathbf{x}\right)  \right)  \epsilon_{\left(  \mathbf{X}%
^{i,N}\left(  s\right)  ,\mathbf{V}^{i,N}\left(  s\right)  \right)
}(d\mathbf{x}\times d\mathbf{v})\,ds\nonumber\label{anastomosis_1}\\
&  =\gamma N h\left(  \frac{1}{N}\left(
K_{2}\ast\delta_{\mathbf{X}^{N}\left(
s\right)  }\right)  \left(  \mathbf{x}\right)  \right)  Q_{N}(s)(d\mathbf{x}%
,d\mathbf{v})ds,
\end{align}
where $\gamma$ is a suitable constant, and $K_{2}:\mathbb{R}^{d}%
\rightarrow\mathbb{R}$ is a suitable mollifying kernel, $h:\mathbb{R}%
_{+}\rightarrow\mathbb{R}_{+}$ is a saturating function of the
form $h\left( r\right)  =\frac{r}{1+r}$. This compensator
expresses the death rate of a tip located at $\left(
\mathbf{X}^{i,N}\left( s\right)  ,\mathbf{V}^{i,N}\left( s\right)
\right)  $ at time $s$; the death rate is modulated by $\gamma$
and by a scaled thickened version of the capillary network
existing at time $s,$ given by (see Equation
(\ref{stoch_network}))
\begin{align}
\frac{1}{N}\left(  K_{2}\ast\delta_{\mathbf{X}^{N}\left(  s\right)  }\right)
\left(  \mathbf{x}\right)   &  =\int_{0}^{s}\frac{1}{N}\sum_{i=1}^{N_{r}}%
K_{2}\left(  \mathbf{x}-\mathbf{X}^{i,N}\left(  r\right)  \right)  \left\vert
\mathbf{V}^{i,N}\left(  r\right)  \right\vert \mathbb{I}_{r\in\lbrack
T^{i,N},\Theta^{i,N})}dr\nonumber\label{anastomosis_2}\\
&  =\int_{0}^{s}\int_{\mathbb{R}^{d}\times\mathbb{R}^{d}}K_{2}\left(
\mathbf{x}-\mathbf{x}^{\prime}\right)  \left\vert \mathbf{v}^{\prime
}\right\vert Q_{N}\left(  r\right)  \left(  d\mathbf{x}^{\prime}%
,d\mathbf{v}^{\prime}\right)  dr.\nonumber
\end{align}
 Let us set%
\[
g(s,\mathbf{x},\{Q_{N}(r)\}_{r\in\lbrack0,s]}):=h\left(  \frac{1}{N}\left(
K_{2}\ast\delta_{\mathbf{X}^{N}\left(  s\right)  }\right)  \left(
\mathbf{x}\right)  \right)
\]

Thanks to the above, the compensator (\ref{anastomosis_1}) can be rewritten
as
\begin{equation}
\gamma Ng(s,\mathbf{x},\{Q_{N}(r)\}_{r\in\lbrack0,s]})Q_{N}\left(  s\right)
\left(  d\mathbf{x},d\mathbf{v}\right)  ds.
\end{equation}

Here we wish to  stress a couple of  technical issues which have
led to the  substantial modification of the structure of the
compensator with respect to previous models (see e.g.
\cite{BCACT}). The first one is mainly motivated by the case of
dimension $d=3$, but then for simplicity we adopt it also in
$d=2;$  since, for mathematical convenience, we have modelled a
vessel as a $1-$dimensional curve in $\mathbb{R}^{d}$, it is
essentially impossible that anastomosis takes place, since the
probability that two curves meet in $\mathbb{R}^{3}$ is
negligible, even though they may get very close to each other: the
mathematical abstraction \textquotedblleft
vessel=curve\textquotedblright\ would have not been realistic
here. In order to overcome this technical issue, we have
introduced a thickness of the curve, described by a kernel $K_{2}$
(this is equivalent to keep vessels as curves and introduce a
thickness of tips). With this technical modification, the model
has become more realistic, since real vessels do not have
dimension $1.$  Anyhow this choice  has implied a second issue.
The thickening of  vessels induce a mathematical modelling problem
whenever the vessel network is highly dense in space; indeed in
such a situation at   a same point $\mathbf{x}$ more than one
vessel may contribute to the quantity

\begin{equation}
\frac{1}{N}\left(  K_{2}\ast\delta_{\mathbf{X}^{N}\left( s\right)}
\right ) (\mathbf{x})
\end{equation}
which is not  realistic. In order  to compete with   this
anomalous effect, we have introduced a saturation via  the
function  $h.$

Thanks to the above  considerations, on one hand we have solved
significant  modelling biases, on the other  hand we have made the
model more tractable from  a mathematical point of view.

\section{Evolution of the empirical measure}

\label{sec:evolution_empirical}

The evolution of the empirical measure $Q_{N}\left(  t\right)  $ is obtained
by application of It\^{o} formula to the expression
\[
\frac{1}{N}\sum_{i=1}^{N_{t}}\mathbb{I}_{t\in\lbrack T^{i,N},\Theta^{i,N}%
)}\phi\left(  \mathbf{X}^{i,N}\left(  t\right)  ,\mathbf{V}^{i,N}\left(
t\right)  \right)
\]
where $\phi$ is a $C^{2}$ test function.

From It\^{o}-Levy formula and the expressions of the compensators
of the branching
and anastomosis processes, we obtain the identity%

\begin{eqnarray} \label{Ito_Levy}
&& \int_{\mathbb{R}^{d}\times\mathbb{R}^{d}}\phi\left(  \mathbf{x},\mathbf{v}%
\right)  Q_{N}\left(  t\right)  \left(  d\mathbf{x},d\mathbf{v}\right)
=\int_{\mathbb{R}^{d}\times\mathbb{R}^{d}}\phi\left(  \mathbf{x}%
,\mathbf{v}\right)  Q_{N}\left(  0\right)  \left(  d\mathbf{x},d\mathbf{v}%
\right) \nonumber \\
 &+& \int_{0}^{t}\int_{\mathbb{R}^{d}\times\mathbb{R}^{d}}\mathbf{v}\cdot
\nabla_{x}\phi\left(  \mathbf{x},\mathbf{v}\right)  Q_{N}\left(
s\right) \left(  d\mathbf{x},d\mathbf{v}\right)  ds  \nonumber \\
&+& \int_{0}^{t}\int_{\mathbb{R}^{d}\times\mathbb{R}^{d}}\left[
f\left( |\nabla C_{N}\left(  s,\mathbf{x}\right)  |\right)  \nabla
C_{N}\left( s,\mathbf{x}\right)  -k_{1}\mathbf{v}\right]
\cdot\nabla_{v}\phi\left(
\mathbf{x},\mathbf{v}\right)  Q_{N}\left(  s\right)  \left(  d\mathbf{x}%
,d\mathbf{v}\right)  ds  \nonumber \\
&+& \int_{0}^{t}\int_{\mathbb{R}^{d}\times\mathbb{R}^{d}}\frac{\sigma^{2}}%
{2}\Delta_{v}\phi\left(  \mathbf{x},\mathbf{v}\right)  Q_{N}\left(  s\right)
\left(  d\mathbf{x},d\mathbf{v}\right)  ds
\nonumber \\
&+&
\int_{0}^{t}\int_{\mathbb{R}^{d}\times\mathbb{R}^{d}}\phi_{G}\left(
\mathbf{x}\right)  \alpha(C_{N}(s,\mathbf{x}))Q_{N}\left( s\right)
\left(
d\mathbf{x},d\mathbf{v}\right)  ds \nonumber \\
&+&
\int_{0}^{t}\int_{\mathbb{R}^{d}\times\mathbb{R}^{d}}\phi_{G}\left(
\mathbf{x}\right)  \beta(C_{N}(s,\mathbf{x}))\left\vert
\mathbf{v}\right\vert \int_{0}^{s}Q_{N}\left(  r\right)  \left(
d\mathbf{x},d\mathbf{v}\right)
drds \nonumber \\
&-&
\gamma\int_{0}^{t}\int_{\mathbb{R}^{d}\times\mathbb{R}^{d}}\phi\left(
\mathbf{x},\mathbf{v}\right)  g\left(  s,\mathbf{x},\left\{
Q_{N}\left( r\right)  \right\}  _{r\in\left[  0,s\right]  }\right)
Q_{N}\left(  s\right) \left(  d\mathbf{x},d\mathbf{v}\right)  ds
\nonumber \\ &+& \widetilde{M}_{N}\left(  t\right).
\end{eqnarray}

The martingale
\[
\widetilde{M}_{N}\left(  t\right)  =\widetilde{M}_{1,N}\left(  t\right)
+\widetilde{M}_{2,N}\left(  t\right)  +\widetilde{M}_{3,N}\left(  t\right)
\]
is the sum of three zero-mean martingales, namely%

\begin{equation}
\widetilde{M}_{1,N}\left(  t\right)  =\int_{0}^{t}\frac{1}{N}\sum_{i=1}%
^{N_{s}}\nabla_{v}\phi(X^{i,N}(s),V^{i,N}(s))I_{s\in\lbrack T^{i,N}%
,\Theta^{i,N})}\cdot dW^{i}(s); \label{martingale_1}%
\end{equation}

\begin{align}
\widetilde{M}_{2,N}\left(  t\right)   &  =\int_{0}^{t}\int_{\mathbb{R}%
^{d}\times\mathbb{R}^{d}}[\phi(\mathbf{x},\mathbf{v})\Phi_{N}(ds\times
d\mathbf{x}\times d\mathbf{v})\nonumber\label{martingale_2}\\
&  -\phi_{G}\left(  \mathbf{x}\right)  \alpha(C_{N}(s,\mathbf{x}))Q_{N}\left(
s\right)  \left(  d\mathbf{x},d\mathbf{v}\right)  ds\nonumber\\
&  -\phi_{G}\left(  \mathbf{x}\right)
\beta(C_{N}(s,\mathbf{x}))\left\vert \mathbf{v}\right\vert
\int_{0}^{s}Q_{N}\left(  r\right)  \left(
d\mathbf{x},d\mathbf{v}\right)  drds];
\end{align}

\begin{align}
\widetilde{M}_{3,N}\left(  t\right)   &  =\int_{0}^{t}\int_{\mathbb{R}%
^{d}\times\mathbb{R}^{d}}\phi(\mathbf{x},\mathbf{v})[\Psi_{N}(ds\times
d\mathbf{x}\times d\mathbf{v})\nonumber\label{martingale_3}\\
&  -\gamma g\left(  s,\mathbf{x},\left\{  Q_{N}\left(  r\right)  \right\}
_{r\in\left[  0,s\right]  }\right)  Q_{N}\left(  s\right)  \left(
d\mathbf{x},d\mathbf{v}\right)  ds].
\end{align}

In the above we have denoted
\begin{equation}
\phi_{G}\left(  \mathbf{x}\right)  :=\int_{\mathbb{R}^{d}}G_{{v}_{0}%
}(\mathbf{v})\phi\left(  \mathbf{x},\mathbf{v}\right)  d\mathbf{v}.
\label{G_V_0}%
\end{equation}

\subsection{Heuristic derivation of the limit PDE\label{subsect limit PDE}}

It is now clear that the only source of stochasticity in the above
system is in the martingale terms. Classical laws of large numbers
for martingales, allow us to conjecture that the martingales are
negligible. Consequently, if we assume we already know that the
sequences $(Q_{N})_{N\in\mathbb{N}}$ and
$(C_{N})_{N\in\mathbb{N}}$ converge, to a deterministic
time-dependent measure $p_{t}\left( d\mathbf{x},d\mathbf{v}\right)
$ and a deterministic function $C_{t}\left(  \mathbf{x}\right)  $
respectively,   the
limit PDE\ for the measure $p_{t}$ is conjectured to be%

\begin{eqnarray}\label{PDE weak form}
&& \int_{\mathbb{R}^{d}\times\mathbb{R}^{d}}\phi\left(
\mathbf{x},\mathbf{v}
\right)  p_{t}\left(  d\mathbf{x},d\mathbf{v}\right)  =\int_{\mathbb{R}%
^{d}\times\mathbb{R}^{d}}\phi\left(  \mathbf{x},\mathbf{v}\right)
p_{0}\left(  d\mathbf{x},d\mathbf{v}\right)  \nonumber \\
&+&
\int_{0}^{t}\int_{\mathbb{R}^{d}\times\mathbb{R}^{d}}\mathbf{v}\cdot
\nabla_{x}\phi\left(  \mathbf{x},\mathbf{v}\right)  p_{s}\left(
d\mathbf{x},d\mathbf{v}\right)  ds \nonumber \\
&+& \int_{0}^{t}\int_{\mathbb{R}^{d}\times\mathbb{R}^{d}}\left[
f\left( |\nabla C_{s}\left(  \mathbf{x}\right)  |\right)  \nabla
C_{s}\left( \mathbf{x}\right)  -k_{1}\mathbf{v}\right]
\cdot\nabla_{v}\phi\left( \mathbf{x},\mathbf{v}\right) p_{s}\left(
d\mathbf{x},d\mathbf{v}\right)
ds \nonumber \\
&+& \int_{0}^{t}\int_{\mathbb{R}^{d}\times\mathbb{R}^{d}}\frac{\sigma^{2}}%
{2}\Delta_{v}\phi\left(  \mathbf{x},\mathbf{v}\right)  p_{s}\left(
d\mathbf{x},d\mathbf{v}\right)  ds
\nonumber \\
&+&
\int_{0}^{t}\int_{\mathbb{R}^{d}\times\mathbb{R}^{d}}\phi_{G}\left(
\mathbf{x}\right)  \alpha(C_{s}\left(  \mathbf{x}\right)
)p_{s}\left(
d\mathbf{x},d\mathbf{v}\right)  ds \nonumber \\
&+&
\int_{0}^{t}\int_{0}^{s}\int_{\mathbb{R}^{d}\times\mathbb{R}^{d}}\phi
_{G}\left(  \mathbf{x}\right)  \beta(C_{s}\left( \mathbf{x}\right)
)\left\vert \mathbf{v}\right\vert p_{r}\left(
d\mathbf{x},d\mathbf{v}\right) drds  \nonumber \\
&-&
\gamma\int_{0}^{t}\int_{\mathbb{R}^{d}\times\mathbb{R}^{d}}\phi\left(
\mathbf{x},\mathbf{v}\right)  g\left(  s,\mathbf{x},\left\{
p_{r}\right\}
_{r\in\left[  0,s\right]  }\right)  p_{s}\left(  d\mathbf{x},d\mathbf{v}%
\right)  ds
\end{eqnarray}
with%
\begin{align*}
g\left(  s,\mathbf{x},\left\{  p_{r}\right\}  _{r\in\left[  0,s\right]
}\right)   &  =h\left(  \int_{0}^{s}\int_{\mathbb{R}^{d}\times\mathbb{R}^{d}%
}K_{2}\left(  \mathbf{x}-\mathbf{x}^{\prime}\right)  \left\vert \mathbf{v}%
^{\prime}\right\vert p_{r}\left(  d\mathbf{x}^{\prime},d\mathbf{v}^{\prime
}\right)  dr\right) \\
&  =h\left(  \int_{0}^{s}\int_{\mathbb{R}^{d}}K_{2}\left(  \mathbf{x}%
-\mathbf{x}^{\prime}\right)  \widetilde{p}_{r}\left(  d\mathbf{x}^{\prime
}\right)  dr\right)  ,
\end{align*}
having  set%
\[
\widetilde{p}_{r}\left(  d\mathbf{x}\right)  =\int_{\mathbb{R}^{d}}\left\vert
\mathbf{v}\right\vert p_{r}\left(  d\mathbf{x},d\mathbf{v}\right)  .
\]

Notice that%
\begin{align*}
&  \int_{\mathbb{R}^{d}\times\mathbb{R}^{d}}\phi_{G}\left(  \mathbf{x}\right)
\alpha(C_{s}\left(  \mathbf{x}\right)  )p_{s}\left(  d\mathbf{x}%
,d\mathbf{v}\right)  =\\
&  \qquad=\int_{\mathbb{R}^{d}\times\mathbb{R}^{d}\times\mathbb{R}^{d}}%
G_{{v}_{0}}(\mathbf{v}^{\prime})\phi\left(  \mathbf{x},\mathbf{v}^{\prime
}\right)  \alpha(C_{s}\left(  \mathbf{x}\right)  )p_{s}\left(  d\mathbf{x}%
,d\mathbf{v}\right)  d\mathbf{v}^{\prime}\\
&  \qquad=\int_{\mathbb{R}^{d}\times\mathbb{R}^{d}\times\mathbb{R}^{d}}%
G_{{v}_{0}}(\mathbf{v})\phi\left(  \mathbf{x},\mathbf{v}\right)  \alpha
(C_{s}\left(  \mathbf{x}\right)  )p_{s}\left(  d\mathbf{x},d\mathbf{v}%
^{\prime}\right)  d\mathbf{v}\\
&  \qquad=\int_{\mathbb{R}^{d}\times\mathbb{R}^{d}}G_{{v}_{0}}(\mathbf{v}%
)\phi\left(  \mathbf{x},\mathbf{v}\right)  \alpha(C_{s}\left(  \mathbf{x}%
\right)  )\left(  \int_{\mathbb{R}^{d}}p_{s}\left(  d\mathbf{x},d\mathbf{v}%
^{\prime}\right)  \right)  d\mathbf{v}\\
&  \qquad=\int_{\mathbb{R}^{d}\times\mathbb{R}^{d}}G_{{v}_{0}}(\mathbf{v}%
)\phi\left(  \mathbf{x},\mathbf{v}\right)  \alpha(C_{s}\left(  \mathbf{x}%
\right)  )\left(  \pi_{1}p_{s}\right)  \left(  d\mathbf{x}\right)  d\mathbf{v}%
\end{align*}
where we set%
\[
\left(  \pi_{1}p_{s}\right)  \left(  d\mathbf{x}\right)  =\int_{\mathbb{R}%
^{d}}p_{s}\left(  d\mathbf{x},d\mathbf{v}\right)
\]
and similarly%
\begin{align*}
&  \int_{0}^{s}\int_{\mathbb{R}^{d}\times\mathbb{R}^{d}}\phi_{G}\left(
\mathbf{x}\right)  \beta(C_{s}\left(  \mathbf{x}\right)  )\left\vert
\mathbf{v}\right\vert p_{r}\left(  d\mathbf{x},d\mathbf{v}\right)  dr\\
&  =\int_{0}^{s}\int_{\mathbb{R}^{d}\times\mathbb{R}^{d}\times\mathbb{R}^{d}%
}G_{{v}_{0}}(\mathbf{v}^{\prime})\phi\left(  \mathbf{x},\mathbf{v}^{\prime
}\right)  \beta(C_{s}\left(  \mathbf{x}\right)  )\left\vert \mathbf{v}%
\right\vert p_{r}\left(  d\mathbf{x},d\mathbf{v}\right)  drd\mathbf{v}%
^{\prime}\\
&  =\int_{0}^{s}\int_{\mathbb{R}^{d}\times\mathbb{R}^{d}\times\mathbb{R}^{d}%
}G_{{v}_{0}}(\mathbf{v})\phi\left(  \mathbf{x},\mathbf{v}\right)  \beta
(C_{s}\left(  \mathbf{x}\right)  )\left\vert \mathbf{v}^{\prime}\right\vert
p_{r}\left(  d\mathbf{x},d\mathbf{v}^{\prime}\right)  drd\mathbf{v}\\
&  =\int_{\mathbb{R}^{d}\times\mathbb{R}^{d}}G_{{v}_{0}}(\mathbf{v}%
)\phi\left(  \mathbf{x},\mathbf{v}\right)  \beta(C_{s}\left(  \mathbf{x}%
\right)  )\int_{0}^{s}\int_{\mathbb{R}^{d}}\left\vert \mathbf{v}^{\prime
}\right\vert p_{r}\left(  d\mathbf{x},d\mathbf{v}^{\prime}\right)
drd\mathbf{v}\\
&  =\int_{\mathbb{R}^{d}\times\mathbb{R}^{d}}G_{{v}_{0}}(\mathbf{v}%
)\phi\left(  \mathbf{x},\mathbf{v}\right)  \beta(C_{s}\left(  \mathbf{x}%
\right)  )\int_{0}^{s}\widetilde{p}_{r}\left(  d\mathbf{x}\right)
drd\mathbf{v}.
\end{align*}

Consequently, the limit PDE for $C_{t}\left(  \mathbf{x}\right)  $ is
conjectured to be%
\[
\frac{\partial}{\partial t}C_{t}\left(  \mathbf{x}\right)  =k_{2}\delta
_{A}\left(  \mathbf{x}\right)  +d_{1}\Delta C_{t}\left(  \mathbf{x}\right)
-\eta\left(  t,\mathbf{x},\left\{  p_{s}\right\}  _{s\in\left[  0,t\right]
}\right)  C_{t}\left(  \mathbf{x}\right)  ,
\]
where
\begin{equation}
\eta\left(  t,\mathbf{x},\left\{  p_{s}\right\}  _{s\in\left[  0,t\right]
}\right)  =\int_{0}^{t}\left(  \int_{\mathbb{R}^{d}\times\mathbb{R}^{d}}%
K_{1}\left(  \mathbf{x}-\mathbf{x}^{\prime}\right)  \left\vert \mathbf{v}%
^{\prime}\right\vert p_{s}\left(  d\mathbf{x}^{\prime},d\mathbf{v}^{\prime
}\right)  \right)  ds. \label{eta}%
\end{equation}

\section{Main results} \label{sec:main_results}

A rigorous proof of the above mentioned convergence of the evolution equations
for the empirical measure $Q_{N}(t)$ to the evolution equation of the
corresponding deterministic limit measure $p_{t}$ requires various steps,
including (i) tightness of the sequences of the laws of $(Q_{N})_{N\in
\mathbb{N}},$ and $(C_{N})_{N\in\mathbb{N}};$ (ii) existence and uniqueness of
the solution of the deterministic evolution equation of the limiting measure
$p_{t}$ (see e.g. \cite{capasso_bakstein} and references therein).

\subsection{Assumptions and notations}

Denote by $\mathcal{M}_{+}\left(
\mathbb{R}^{d}\times\mathbb{R}^{d}\right)  $ the space of positive
Radon measures and by $\mathcal{M}_{1}\left(
\mathbb{R}^{d}\times\mathbb{R}^{d}\right)  $ the space of those
$\rho\left( d\mathbf{x},d\mathbf{v}\right)  $ such that
$$\int_{\mathbb{R}^{2d}}\left(
1+\left\vert \mathbf{v}\right\vert \right)  \rho\left(  d\mathbf{x}%
,d\mathbf{v}\right)  <\infty.$$
Denote by $L^{\infty}\left(  0,T;\mathcal{M}%
_{1}\left(  \mathbb{R}^{d}\times\mathbb{R}^{d}\right)  \right)  $ the space of
time-dependent Radon measures $p_{t}\left(  d\mathbf{x},d\mathbf{v}\right)  $
such that $t\mapsto\int_{\mathbb{R}^{2d}}\phi\left(  \mathbf{x}%
,\mathbf{v}\right)  \left(  1+\left\vert \mathbf{v}\right\vert \right)
p_{t}\left(  d\mathbf{x},d\mathbf{v}\right)  $ is measurable for all $\phi\in
C_{c}^{\infty}\left(  \mathbb{R}^{d}\times\mathbb{R}^{d}\right)  $ and
$$\sup_{t\in\left[  0,T\right]  }\int_{\mathbb{R}^{2d}}\left(  1+\left\vert
\mathbf{v}\right\vert \right)  p_{t}\left(
d\mathbf{x},d\mathbf{v}\right) <\infty.$$

Denote by $C\left(  \left[  0,T\right]  ;\mathcal{M}_{+}\left(
\mathbb{R}^{d}\times\mathbb{R}^{d}\right)  \right)  $ the space of
time-dependent measures that are continuous when
$\mathcal{M}_{+}\left( \mathbb{R}^{d}\times\mathbb{R}^{d}\right)
$ is endowed of a metric corresponding to weak convergence of
measures.

For a positive integer $k$, we denote by $C_{b}^{k}\left(  \mathbb{R}%
^{d}\right)  $ the space of all functions on $\mathbb{R}^{d}$ which are
differentiable $k$ times with bounded derivatives up to order $k$. We denote
by $UC_{b}^{k}\left(  \mathbb{R}^{d}\right)  $ the subspace of $C_{b}%
^{k}\left(  \mathbb{R}^{d}\right)  $ of functions which are uniformly
continuous, with their derivatives up to order $k$.

We assume that the initial conditions $\left(  \mathbf{X}_{0}^{N}%
,\mathbf{V}_{0}^{N}\right)  $ are i.i.d. random vectors, with a compact
support law $p_{0}\in\mathcal{M}_{1}\left(  \mathbb{R}^{d}\times\mathbb{R}%
^{d}\right)  $. Recall the definition of the empirical measure $Q_{N}\left(
t\right)  $ given in (\ref{empirical meas}). From the previous assumption on
the initial conditions $\left(  \mathbf{X}_{0}^{N},\mathbf{V}_{0}^{N}\right)
$ we deduce that $Q_{N}\left(  0\right)  $ converges to $p_{0}$ in the
following sense:%

\begin{align*}
&  \int_{\mathbb{R}^{d}\times\mathbb{R}^{d}}\phi\left(  \mathbf{x}%
,\mathbf{v}\right)  \left(  1+\left\vert \mathbf{v}\right\vert \right)
Q_{N}\left(  0\right)  \left(  d\mathbf{x},d\mathbf{v}\right)  \rightarrow\\
&  \qquad\qquad\rightarrow\int_{\mathbb{R}^{d}\times\mathbb{R}^{d}}\phi\left(
\mathbf{x},\mathbf{v}\right)  \left(  1+\left\vert \mathbf{v}\right\vert
\right)  p_{0}\left(  d\mathbf{x},d\mathbf{v}\right)
\end{align*}
for every $\phi\in C_{b}\left(  \mathbb{R}^{d}\times\mathbb{R}^{d}\right)  $,
where the convergence is understood in probability.

The initial condition $C_{0}$ of the concentration is independent of $N$, just
for simplicity. We assume it of class $UC_{b}^{2}\left(  \mathbb{R}%
^{d}\right)  $. Moreover we assume%
\[
0\leq C_{0}\left(  \mathbf{x}\right)  \leq C_{\max},
\]
for some constant $C_{\max}>0$. The convolution kernel $K_{1}$ appearing in
the TAF absorption rate $\eta$ and the convolution kernel $K_{2}$ of the
anastomosis, are both assumed of class $UC_{b}^{1}\left(  \mathbb{R}%
^{d}\right)  $ and nonnegative%
\[
K_{1}\left(  \mathbf{x}\right)  \geq0,\qquad K_{2}\left(  \mathbf{x}\right)
\geq0.
\]
for all $x\in\mathbb{R}^{d}$. The function $G_{v_{0}}\left(  \mathbf{v}%
\right)  $ appearing in the vessel branching rate is assumed of class
$UC_{b}^{1}\left(  \mathbb{R}^{d}\right)  $, with compact support, non
negative, and such that
\[
\int_{\mathbb{R}^{d}}G_{v_{0}}\left(  \mathbf{v}\right)  \left(  1+\left\vert
\mathbf{v}\right\vert \right)  d\mathbf{v}<\infty.
\]
For $\delta_{A}$ we take its mollified version, so that we may assume that it
is a classical function of class $UC_{b}^{1}\left(  \mathbb{R}^{d}\right)  .$

Several constants appear in the model; we assume:
\[
k_{1},k_{2},d_{1},d_{2},\gamma_{1},\alpha_{1},C_{R}>0.
\]

\subsection{Theorems of convergence and well posedness of the limit PDE
system}

Under these assumptions, we prove our main result.

\begin{theorem}
\label{thm convergence} As $N\rightarrow\infty,$ $Q_{N}\left(  t\right)  $
converges in probability, in the topology $C\left(  \left[  0,T\right]
;\mathcal{M}_{+}\left(  \mathbb{R}^{d}\times\mathbb{R}^{d}\right)  \right)  $,
to a time dependent deterministic Radon measure $p_{t}$ on $\mathbb{R}%
^{d}\times\mathbb{R}^{d}$, also of class $L^{\infty}\left(  0,T;\mathcal{M}%
_{1}\left(  \mathbb{R}^{d}\times\mathbb{R}^{d}\right)  \right)  $ and $C_{N}$
converges in $C\left(  \left[  0,T\right]  ;C_{b}^{1}\left(  \mathbb{R}%
^{d}\right)  \right)  $ to a deterministic function $C$ of this space. The
measure $p_{t}$ is a weak solution (unique in $L^{\infty}\left(
0,T;\mathcal{M}_{1}\left(  \mathbb{R}^{d}\times\mathbb{R}^{d}\right)  \right)
$ as specified in Theorem (\ref{thm uniqueness}) below) of the equation%
\begin{align}
&  \partial_{t}p_{t}+\mathbf{v}\cdot\nabla_{x}p_{t}+\operatorname{div}%
_{v}\left(  \left[  f\left(  |\nabla C_{t}|\right)  \nabla C_{t}%
-k_{1}\mathbf{v}\right]  p_{t}\right)  =\nonumber\\
&  \qquad=\frac{\sigma^{2}}{2}\Delta_{v}p_{t}\nonumber\\
&  \qquad\qquad+G_{{v}_{0}}\left(  \mathbf{v}\right)  d\mathbf{v}\left(
\alpha(C_{t})\left(  \pi_{1}p_{t}\right)  \left(  d\mathbf{x}\right)
+\beta(C_{t})\int_{0}^{t}\widetilde{p}_{r}\left(  d\mathbf{x}\right)
dr\right) \nonumber\\
&  \qquad\qquad-\gamma p_{t}h\left(  \int_{0}^{t}\left(  K_{2}\ast
\widetilde{p}_{r}\right)  \left(  \mathbf{x}\right)  dr\right)  .
\label{PDE for p version 0}%
\end{align}
The function $C_{t}$ is a mild solution of the equation
\begin{equation}
\partial_{t}C_{t}\left(  \mathbf{x}\right)  =k_{2}\delta_{A}\left(
\mathbf{x}\right)  +d_{1}\Delta C_{t}\left(  \mathbf{x}\right)  -\eta\left(
t,\mathbf{x},\left\{  p_{s}\right\}  _{s\in\left[  0,t\right]  }\right)
C_{t}\left(  \mathbf{x}\right)  , \label{PDE for C version 0}%
\end{equation}
where $\eta$ is given in Equation (\ref{eta}).
\end{theorem}

The notion of weak solution of equation (\ref{PDE for p version 0}) is
identity (\ref{PDE weak form}). The notion of mild solution of Equation
(\ref{PDE for C version 0}) is explained in Section \ref{sct uniqueness} (see
also Section \ref{sect assumptions}).

As anticipated above, the proof of Theorem \ref{thm convergence} is based on
several arguments including a uniqueness result for the system of PDEs
(\ref{PDE for p version 0})-(\ref{PDE for C version 0}), which we state
separately because of its independent interest.

\begin{theorem}
\label{thm uniqueness} There exists a unique solution of System
(\ref{PDE for p version 0})-(\ref{PDE for C version 0}), with $p\in L^{\infty
}\left(  0,T;\mathcal{M}_{1}\left(  \mathbb{R}^{d}\times\mathbb{R}^{d}\right)
\right)  $ and $C\in C\left(  \left[  0,T\right]  ;C_{b}^{1}\left(
\mathbb{R}^{d}\right)  \right)  $.
\end{theorem}

The proof of Theorem \ref{thm convergence} is given in Section
\ref{sct proof thm convergence}. The proof of Theorem \ref{thm uniqueness} is
given in Section \ref{sct uniqueness}.

\begin{remark}
\label{density} In the framework of this paper the PDE for the measure $p_{t}$
will be always interpreted as a PDE for measure-valued functions. However,
under suitable assumptions, the relevant measure may admit a density $\rho
_{t}(x,v)$ which is a classical function, so that the PDE can be interpreted
in a more classical sense (we do not investigate rigorously this issue here,
we only give the heuristic result). The formal expression for the evolution
equation of the density $\rho_{t}(x,v)$ would then be
\begin{align}
&  \partial_{t}\rho_{t}\left(  \mathbf{x},\mathbf{v}\right)  +\mathbf{v}%
\cdot\nabla_{x}\rho_{t}\left(  \mathbf{x},\mathbf{v}\right)
+\operatorname{div}_{v}\left(  \left[  f\left(  |\nabla C_{t}\left(
\mathbf{x}\right)  |\right)  \nabla C_{t}\left(  \mathbf{x}\right)
-k_{1}\mathbf{v}\right]  \rho_{t}\left(  \mathbf{x},\mathbf{v}\right)
\right)  \nonumber\label{density_equation}\\
&  =\frac{\sigma^{2}}{2}\Delta_{v}\rho_{t}\left(  \mathbf{x},\mathbf{v}%
\right)  +G_{{v}_{0}}\left(  \mathbf{v}\right)  \left(  \alpha(C_{t}\left(
\mathbf{x}\right)  )\left(  \pi_{1}\rho_{t}\right)  \left(  \mathbf{x}\right)
+\beta(C_{t}\left(  \mathbf{x}\right)  )\int_{0}^{t}\widetilde{\rho}%
_{r}\left(  \mathbf{x}\right)  dr\right)  \nonumber\\
&  -\gamma\rho_{t} h\left(\int_{0}^{t} \left(
K_{2}\ast\widetilde{\rho}_{r}\right) \left(  \mathbf{x}\right) dr
\right ).
\end{align}
Here we have taken
\[
\left(  \pi_{1}\rho_{t}\right)  \left(  \mathbf{x}\right)  =\int%
_{\mathbb{R}^{d}}\rho_{t}\left(  \mathbf{x},\mathbf{v}\right)  d\mathbf{v},
\]
and
\[
\widetilde{\rho}_{r}\left(  \mathbf{x}\right)  =\int_{\mathbb{R}^{d}%
}|\mathbf{v}|\rho_{r}\left(  \mathbf{x},\mathbf{v}\right)  d\mathbf{v.}%
\]
\end{remark}

\section{Proof of Theorem \ref{thm convergence}%
\label{sct proof thm convergence}}

Let us explain the steps of the proof. First, we prove bounds, uniform in $N$,
on the particle system (\ref{vessel_extension}) and the PDE (\ref{eq CN}).
This is the core of the method. From these bounds we deduce tightness of the
sequence of laws of $Q_{N}$, $N\in\mathbb{N}$ and $C_{N}$, $N\in\mathbb{N}$,
and therefore the existence of a weakly convergent subsequence. Then we show
that the limit of this subsequence is concentrated on solutions of the limit
system (\ref{PDE for p version 0})-(\ref{PDE for C version 0}). This provides,
in particular, the existence claim of Theorem \ref{thm uniqueness}. From the
uniqueness claim of that theorem, proved in Section \ref{sct uniqueness}
below, we deduce that the whole sequence $\left(  Q_{N},C_{N}\right)  $,
$N\in\mathbb{N},$ converges weakly; and converges also in probability because
the limit is deterministic (again due to uniqueness).

\subsection{Regularity of $\eta$ and $C_{N}$\label{sect assumptions}}

We interpret equation (\ref{eq CN}) for $C_{N}$ in the mild semigroup form%
\[
C_{N}\left(  t\right)  =e^{tA}C_{0}+\int_{0}^{t}e^{\left(  t-s\right)
A}\left(  k_{2}\delta_{A}-\eta_{N}\left(  s,\cdot,\left\{  Q_{N}\left(
r\right)  \right\}  _{r\in\left[  0,s\right]  }\right)  C_{N}\left(  s\right)
\right)  ds.
\]
Here $e^{tA}$ denotes the heat semigroup associated to the operator
$A:=d_{1}\Delta$ on $\mathbb{R}^{d};$ $A:UC_{b}^{2}\left(  \mathbb{R}%
^{d}\right)  \subset UC_{b}^{0}\left(  \mathbb{R}^{d}\right)  \rightarrow
UC_{b}^{0}\left(  \mathbb{R}^{d}\right)  $.

\begin{lemma}
\label{lemma on eta}Given a measure $\mu\in L^{\infty}\left(  0,T;\mathcal{M}%
_{1}\left(  \mathbb{R}^{d}\times\mathbb{R}^{d}\right)  \right)  $ the
function
\[
{\small {\eta\left(  t,\mathbf{x}\right)  :=\eta\left(  t,\mathbf{x},\left\{
\mu_{s}\right\}  _{s\in\left[  0,t\right]  }\right)  =\int_{0}^{t}\left(
\int_{\mathbb{R}^{d}\times\mathbb{R}^{d}}K_{1}\left(  \mathbf{x}%
-\mathbf{x}^{\prime}\right)  \left\vert \mathbf{v}^{\prime}\right\vert \mu
_{s}\left(  d\mathbf{x}^{\prime},d\mathbf{v}^{\prime}\right)  \right)  ds}}%
\]
is of class $C\left(  \left[  0,T\right]  ;UC_{b}^{1}\left(  \mathbb{R}%
^{d}\right)  \right)  $.
\end{lemma}

\begin{proof}
It follows from the assumption $K_{1}\in C_{b}^{1}\left(  \mathbb{R}%
^{d}\right)  $ and repeated application of Lebesgue dominated convergence
theorem and the definition of uniform continuity, applied first to check
continuity, then differentiability, finally uniform continuity of the
derivatives. Boundedness of $\eta_{N}$ and its derivatives comes from the
boundedness of $K_{1}$ and its derivatives and from the bound fulfilled by
elements of $L^{\infty}\left(  0,T;\mathcal{M}_{1}\left(  \mathbb{R}^{d}%
\times\mathbb{R}^{d}\right)  \right)  $.
\end{proof}

The equation for $C_{N}$ is not closed, since it depends on $Q_{N}$ which
depends on $C_{N}$ via (\ref{vessel_extension}). However, let us first
understand the regularity of $C_{N}$ when $Q_{N}$ is given. So, in the next
lemma, the tacit assumption is that $Q_{N}$ is a well defined adapted random
element of $L^{\infty}\left(  0,T;\mathcal{M}_{1}\left(  \mathbb{R}^{d}%
\times\mathbb{R}^{d}\right)  \right)  $.

\begin{corollary}
\label{corollary on CN}$C_{N}$ is an adapted process with paths of class
$C\left(  \left[  0,T\right]  ;UC_{b}^{1}\left(  \mathbb{R}^{d}\right)
\right)  $. Moreover,
\begin{align*}
&  \left\Vert \partial_{i}C_{N}\left(  t\right)  \right\Vert _{\infty}\leq
c\left\Vert \partial_{i}C_{0}\right\Vert _{\infty}\\
&  \quad+\int_{0}^{t}\frac{c}{\sqrt{t-s}}\left(  \left\Vert k_{2}\delta
_{A}\right\Vert _{\infty}+\left\Vert \eta\left(  s,\cdot,\left\{  Q_{N}\left(
r\right)  \right\}  _{r\in\left[  0,s\right]  }\right)  \right\Vert _{\infty
}\left\Vert C_{N}\left(  s\right)  \right\Vert _{\infty}\right)  ds
\end{align*}
for some constant $c>0$.
\end{corollary}

\begin{proof}
It is clear that the sum of the first two terms
\[
w\left(  t\right)  :=e^{tA}C_{0}+\int_{0}^{t}e^{\left(  t-s\right)  A}%
k_{2}\delta_{A}ds
\]
is an element of $C\left(  \left[  0,T\right]  ;UC_{b}^{1}\left(
\mathbb{R}^{d}\right)  \right)  $, since derivatives commute with the heat
semigroup and we use the assumption $C_{0},\delta_{A}\in UC_{b}^{1}\left(
\mathbb{R}^{d}\right)  $. Then, taken a single realization of $\eta_{N}\left(
s,\cdot,\left\{  Q_{N}\left(  r\right)  \right\}  _{r\in\left[  0,s\right]
}\right)  $, thanks to the previous lemma, it is sufficient to apply the
contraction principle to the map%
\[
C_{N}\mapsto\Lambda\left(  C_{N}\right)  \left(  t\right)  :=w\left(
t\right)  -\int_{0}^{t}e^{\left(  t-s\right)  A}\eta_{N}\left(  s,\cdot
,\left\{  Q_{N}\left(  r\right)  \right\}  _{r\in\left[  0,s\right]  }\right)
C_{N}\left(  s\right)  ds
\]
in the space $C\left(  \left[  0,T\right]  ;UC_{b}^{1}\left(  \mathbb{R}%
^{d}\right)  \right)  $ (first locally in time, then on repeated intervals of
equal length). A posteriori, the unique fixed point $C_{N}$ depends measurably
on the randomness, being the limit of iterates which are measurable by direct
construction. To check that $C_{N}$ is adapted it is sufficient to apply the
previous measurability argument to each interval $\left[  0,t\right]  $. Let
us prove the inequality in the claim of the corollary. From the mild
formulation of the PDE for $C_{N}$ we have
\begin{equation}
\partial_{i}C_{N}\left(  t\right)  =e^{tA}\partial_{i}C_{0}+\int_{0}%
^{t}\partial_{i}e^{\left(  t-s\right)  A}\left(  k_{2}\delta_{A}-\eta\left(
s,\cdot,\left\{  Q_{N}\left(  r\right)  \right\}  _{r\in\left[  0,s\right]
}\right)  C_{N}\left(  s\right)  \right)  ds
\end{equation}
where we have used the fact that $\partial_{j}e^{tA}f=e^{tA}\partial_{j}f$ for
every $f\in UC_{b}^{1}\left(  \mathbb{R}^{d}\right)  $. It is well known that
there exists a constant $C>0$ such that
\begin{equation}
\left\Vert \partial_{i}e^{tA}f\right\Vert _{\infty}\leq\frac{C}{\sqrt{t}%
}\left\Vert f\right\Vert _{\infty}\label{heat regulariz}%
\end{equation}
for all $t>0$ and $f\in C_{b}^{0}\left(  \mathbb{R}^{d}\right)  $. The
inequality of the corollary readily follows.
\end{proof}

In fact, due to the regularization properties of the heat semigroup, the paths
of $C_{N}$ are more regular. We express here only one regularity property, not
the maximal one.

\begin{proposition}
\label{proposition on CN}$C_{N}$ has a.e. path of class $C\left(  \left[
0,T\right]  ;UC_{b}^{2}\left(  \mathbb{R}^{d}\right)  \right)  $, and%
\begin{align*}
\left\Vert \partial_{i}\partial_{j}C_{N}\left(  t\right)  \right\Vert
_{\infty}  &  \leq c\left\Vert \partial_{i}\partial_{j}C_{0}\right\Vert
_{\infty}+\int_{0}^{t}\frac{c}{\sqrt{t-s}}k_{2}\left\Vert \partial_{j}%
\delta_{A}\right\Vert _{\infty}ds\\
&  +\int_{0}^{t}\frac{c}{\sqrt{t-s}}\left\Vert \partial_{j}\eta\left(
s,\cdot,\left\{  Q_{N}\left(  r\right)  \right\}  _{r\in\left[  0,s\right]
}\right)  \right\Vert _{\infty}\left\Vert C_{N}\left(  s\right)  \right\Vert
_{\infty}ds\\
&  +\int_{0}^{t}\frac{c}{\sqrt{t-s}}\left\Vert \eta\left(  s,\cdot,\left\{
Q_{N}\left(  r\right)  \right\}  _{r\in\left[  0,s\right]  }\right)
\right\Vert _{\infty}\left\Vert \partial_{j}C_{N}\left(  s\right)  \right\Vert
_{\infty}ds
\end{align*}
for some constant $c>0$.
\end{proposition}

\begin{proof}
From the mild formulation of the PDE for $C_{N}$, as in the previous proof, we
have%
\begin{align}
&  \partial_{i}\partial_{j}C_{N}\left(  t\right)  =e^{tA}\partial_{i}%
\partial_{j}C_{0}\nonumber\\
&  \quad+\int_{0}^{t}\partial_{i}e^{\left(  t-s\right)  A}\partial_{j}\left(
k_{2}\delta_{A}-\eta\left(  s,\cdot,\left\{  Q_{N}\left(  r\right)  \right\}
_{r\in\left[  0,s\right]  }\right)  C_{N}\left(  s\right)  \right)
ds\label{semigr est}%
\end{align}
where we have used the fact that $C_{0}\in UC_{b}^{2}\left(  \mathbb{R}%
^{d}\right)  $ by assumption. We know from the assumption on $\delta_{A}$,
from Lemma \ref{lemma on eta} and Corollary \ref{corollary on CN}, that%
\[
\partial_{j}\left(  k_{2}\delta_{A}-\eta\left(  s,\cdot,\left\{  Q_{N}\left(
r\right)  \right\}  _{r\in\left[  0,s\right]  }\right)  C_{N}\left(  s\right)
\right)
\]
has paths in $C\left(  \left[  0,T\right]  ;UC_{b}^{0}\left(  \mathbb{R}%
^{d}\right)  \right)  $. Hence%
\begin{align*}
&  \left\Vert \partial_{i}e^{\left(  t-s\right)  A}\partial_{j}\left(
k_{2}\delta_{A}-\eta\left(  s,\cdot,\left\{  Q_{N}\left(  r\right)  \right\}
_{r\in\left[  0,s\right]  }\right)  C_{N}\left(  s\right)  \right)
\right\Vert _{\infty}\\
&  \leq\frac{C}{\sqrt{t-s}}\left\Vert \partial_{j}\left(  k_{2}\delta_{A}%
-\eta\left(  s,\cdot,\left\{  Q_{N}\left(  r\right)  \right\}  _{r\in\left[
0,s\right]  }\right)  C_{N}\left(  s\right)  \right)  \right\Vert _{\infty
}\leq\frac{C^{\prime}}{\sqrt{t-s}}%
\end{align*}
for a suitable constant $C^{\prime}>0$. Since $\displaystyle\frac{1}%
{\sqrt{t-s}}$ is integrable on $\left[  0,t\right]  $, it follows that
\[
\int_{0}^{t}\partial_{i}e^{\left(  t-s\right)  A}\partial_{j}\left(
k_{2}\delta_{A}-\eta\left(  s,\cdot,\left\{  Q_{N}\left(  r\right)  \right\}
_{r\in\left[  0,s\right]  }\right)  C_{N}\left(  s\right)  \right)  ds
\]
is an element of $C\left(  \left[  0,T\right]  ;UC_{b}^{0}\left(
\mathbb{R}^{d}\right)  \right)  $. The same holds for $e^{tA}\partial
_{i}\partial_{j}C_{0}$ since $C_{0}\in UC_{b}^{2}\left(  \mathbb{R}%
^{d}\right)  $. Therefore $\partial_{i}\partial_{j}C_{N}$ is in $C\left(
\left[  0,T\right]  ;UC_{b}^{0}\left(  \mathbb{R}^{d}\right)  \right)  $. This
proves the regularity claim. The inequality is obtained by the estimates
explained during the proof.
\end{proof}

Finally, from the property $0\leq C_{0}\left(  x\right)  \leq C_{\max}$, by
classical maximum principle estimates, we deduce

\begin{lemma}
\label{lemma max principle}%
\[
0\leq C_{N}\left(  t,\mathbf{x}\right)  \leq C_{\max}%
\]
for all $t\geq0$ and $x\in\mathbb{R}^{d}$.
\end{lemma}

We have used also the fact that $\eta\left(  t,\mathbf{x},\left\{
Q_{N}\left(  s\right)  \right\}  _{s\in\left[  0,t\right]  }\right)  \geq0$.

\subsection{Preliminary estimates on $C_{N}$ based on $\left\vert
\mathbf{V}^{i,N}\right\vert $}

We summarize the result of the previous section in the following lemma.

\begin{lemma}
There exist constants $a_{0},a_{1},a_{2},a_{3},a_{4}>0$ such that, for
$i,j=1,...,d$,%
\begin{equation}
\left\Vert \partial_{i}C_{N}\left(  t\right)  \right\Vert _{\infty}\leq
a_{0}+\int_{0}^{t}\frac{a_{1}}{\sqrt{t-s}}\int_{0}^{s}\frac{1}{N}\sum
_{i=1}^{N_{r}}1_{r\in\lbrack T^{i,N},\Theta^{i,N})}\left\vert \mathbf{V}%
^{i,N}\left(  r\right)  \right\vert drds \label{est on C N}%
\end{equation}%
\begin{equation}
\left\Vert \partial_{i}\partial_{j}C_{N}\left(  t\right)  \right\Vert
_{\infty}\leq a_{2}+\int_{0}^{t}\frac{a_{3}}{\sqrt{t-s}}\int_{0}^{s}\frac
{1}{N}\sum_{i=1}^{N_{r}}1_{r\in\lbrack T^{i,N},\Theta^{i,N})}\left\vert
\mathbf{V}^{i,N}\left(  r\right)  \right\vert drds \label{est on C N 2}%
\end{equation}%
\[
+\int_{0}^{t}\frac{a_{4}}{\sqrt{t-s}}\left(  \int_{0}^{s}\frac{1}{N}\sum
_{i=1}^{N_{r}}1_{r\in\lbrack T^{i,N},\Theta^{i,N})}\left\vert \mathbf{V}%
^{i,N}\left(  r\right)  \right\vert dr\right)  \left\Vert \partial_{j}%
C_{N}\left(  s\right)  \right\Vert _{\infty}ds.
\]

\end{lemma}

\begin{proof}
Recall that \begin{equation} \small{\eta\left(
s,\mathbf{x},\left\{ Q_{N}\left( r\right)  \right\}
_{r\in\left[  0,s\right]  }\right)  =\int_{0}^{s}\frac{1}{N}\sum_{i=1}^{N_{r}%
}1_{r\in\lbrack T^{i,N},\Theta^{i,N})}K_{1}\left(  \mathbf{x}-\mathbf{X}%
^{i,N}\left(  r\right)  \right)  \left\vert \mathbf{V}^{i,N}\left(
r\right) \right\vert dr.}
\end{equation}

Hence%
\begin{equation} \small{
\left\Vert \eta\left(  s,\mathbf{x},\left\{  Q_{N}\left(  r\right)
\right\} _{r\in\left[  0,s\right]  }\right)  \right\Vert
_{\infty}\leq\left\Vert
K_{1}\right\Vert _{\infty}\int_{0}^{s}\frac{1}{N}\sum_{i=1}^{N_{r}}%
1_{r\in\lbrack T^{i,N},\Theta^{i,N})}\left\vert
\mathbf{V}^{i,N}\left( r\right)  \right\vert dr }
\end{equation}
\begin{equation} \small{
\left\Vert \partial_{j}\eta\left(  s,\mathbf{x},\left\{
Q_{N}\left( r\right)  \right\}  _{r\in\left[  0,s\right]  }\right)
\right\Vert _{\infty
}\leq\left\Vert \partial_{j}K_{1}\right\Vert _{\infty}\int_{0}^{s}\frac{1}%
{N}\sum_{i=1}^{N_{r}}1_{r\in\lbrack
T^{i,N},\Theta^{i,N})}\left\vert \mathbf{V}^{i,N}\left(  r\right)
\right\vert dr. }
\end{equation}
From Lemma \ref{lemma max principle} we have%
\[
\left\Vert C_{N}\right\Vert _{\infty}\leq C_{\max}.
\]
Then, from the inequality of Corollary \ref{corollary on CN} we have%
\begin{align*}
\left\Vert \partial_{i}C_{N}\left(  t\right)  \right\Vert _{\infty} &  \leq
c\left\Vert \partial_{i}C_{0}\right\Vert _{\infty}+\int_{0}^{t}\frac{c}%
{\sqrt{t-s}}\left\Vert k_{2}\delta_{A}\right\Vert _{\infty}ds\\
&  +\int_{0}^{t}\frac{cC_{\max}\left\Vert K_{1}\right\Vert _{\infty}}%
{\sqrt{t-s}}\int_{0}^{s}\frac{1}{N}\sum_{i=1}^{N_{r}}1_{r\in\lbrack
T^{i,N},\Theta^{i,N})}\left\vert \mathbf{V}^{i,N}\left(  r\right)  \right\vert
drds
\end{align*}
hence we have the first inequality of the lemma, taking
\[
a_{0}=c\left\Vert \nabla C_{0}\right\Vert _{\infty}+2ck_{2}\left\Vert
\delta_{A}\right\Vert _{\infty}T^{1/2},\quad  \textnormal{and}\quad a_{1}%
=c\,C_{\max}\left\Vert K_{1}\right\Vert _{\infty}.
\]
Now, from the inequality of Proposition \ref{proposition on CN} we have%
\begin{align*}
&  \left\Vert \partial_{i}\partial_{j}C_{N}\left(  t\right)  \right\Vert
_{\infty}\leq c\left\Vert \partial_{i}\partial_{j}C_{0}\right\Vert _{\infty
}+\int_{0}^{t}\frac{c}{\sqrt{t-s}}k_{2}\left\Vert \partial_{j}\delta
_{A}\right\Vert _{\infty}ds\\
&  \quad+\int_{0}^{t}\frac{cC_{\max}\left\Vert \partial_{j}K_{1}\right\Vert
_{\infty}}{\sqrt{t-s}}\int_{0}^{s}\frac{1}{N}\sum_{i=1}^{N_{r}}1_{r\in\lbrack
T^{i,N},\Theta^{i,N})}\left\vert \mathbf{V}^{i,N}\left(  r\right)  \right\vert
drds\\
&  \quad+\int_{0}^{t}\frac{c}{\sqrt{t-s}}\left\Vert K_{1}\right\Vert _{\infty
}\int_{0}^{s}\frac{1}{N}\sum_{i=1}^{N_{r}}1_{r\in\lbrack T^{i,N},\Theta
^{i,N})}\left\vert \mathbf{V}^{i,N}\left(  r\right)  \right\vert dr\left\Vert
\partial_{j}C_{N}\left(  s\right)  \right\Vert _{\infty}ds.
\end{align*}
Hence we have the second inequality of the lemma, taking
\[
a_{2}=c\left\Vert D^{2}C_{0}\right\Vert
_{\infty}+2ck_{2}\left\Vert \nabla\delta_{A}\right\Vert
_{\infty}T^{1/2},a_{3}=c\,C_{\max}\left\Vert \nabla
K_{1}\right\Vert _{\infty}\quad\textnormal{and}\quad
a_{4}=c\left\Vert K_{1}\right\Vert _{\infty}.
\]
\end{proof}

\subsection{Upper bound on the number of particles \label{section upper bound}%
}

Let us recall that $N_{t}$ denotes the number of active tip cells at time $t$.
Clearly this number depends on the initial number $N$ of tips; we might have
written $N_{t}^{N}$ to emphasize this dependence, but we have preferred to
keep the simpler notation $N_{t}$. In this section we establish bounds on
$N_{t};$ actually we mean bounds on the ratio%
\[
\frac{N_{t}}{N}%
\]
since this is the only quantity that may have bounds (on the average)
independent of $N$.

\begin{theorem}
\label{upper_bound} There exists a $\lambda>0$ such that%
\[
\mathbb{E}\left[  \sup_{t\in\left[  0,T\right]  }\frac{N_{t}}{N}\right]  \leq
e^{\lambda T}%
\]
for all $N\in\mathbb{N}$ and $T\geq0$.
\end{theorem}

\begin{remark}
The rest of this section is devoted to the proof of this result. In the
classical case when branching of particles occurs only at the particle
position, it is usual to introduce a Yule process which dominates the
branching process under study:\ one has to take, as parameter of the Yule
process, any number that bounds from above the variable rates of branching of
the particles, see for instance [21]. In the case of the system studied here
we are faced with two difficulties. The first one is that branching occurs
also along the vessels; there is now a spatial density rate and it is less
easy to relate this variable density rate with a constant upper bound of Yule
type. This is made even more difficult by the presence of the factor
$\left\vert v\right\vert $ in the branching rates, a factor that is a priori
unbounded (see (\ref{vessel_branching}), and (\ref{branching_2})). We thus
have to work much more than in the classical case.
\end{remark}

\subsubsection{Proof of Theorem \ref{upper_bound}}

We may obtain a preliminary domination from above, by considering the same
system without anastomosis. The total number of active tips in the system with
anastomosis is smaller than in the same system without it. It is then
sufficient to obtain a bound the ratio $\displaystyle\frac{N_{t}}{N}$ for the
case $\gamma=0$, $\Theta^{i,N}=+\infty,\,i=1,\ldots,N_{t}$.

With reference to this modified process, denote by
\[
\left(  \mathbf{X}^{i,N}\left(  t\right)  ,\mathbf{V}^{i,N}\left(  t\right)
\right)  _{t\geq T^{i,N}},\,i=1,\ldots,N_{t},
\]
the active tips of this system. Each $i-$tip, for $i=1,\ldots,N_{t},$ is able
to create new tips either by branching at the tip itself, at position
$\mathbf{X}^{i,N}\left(  t\right)  $, or by branching along the vessel
$\left(  \mathbf{X}^{i,N}\left(  s\right)  \right)  _{s\in\left[
T^{i,N},t\right]  }$ that it has generated up to time $t\geq T^{i,N}.$ The
time-rate of creation of new particles, either by $\mathbf{X}^{i,N}\left(
t\right)  $ or by its vessel is obtained by the integral on space of the
relevant space-time rates (\ref{tip_branching}), or (\ref{vessel_branching})
respectively, and it tells us the rate of creations in time, independently of
the position where creation occurs. The time-rate of creation at the tip
position is then given by
\[
\lambda^{i,1}\left(  t\right)  :=\mathbb{I}_{t\geq T^{i,N}}\alpha
(C_{N}(t,\mathbf{X}^{i,N}\left(  t\right)  ))g_{0}%
\]
where $g_{0}=\int_{\mathbb{R}^{d}}G_{\mathbf{v}_{0}}(v)dv$; and the time-rate
of creation along the vessel $\left(  \mathbf{X}^{i,N}\left(  s\right)
\right)  _{s\in\left[  T^{i,N},t\right]  }$ is given by%
\[
\lambda^{i,2}\left(  t\right)  :=\mathbb{I}_{t\geq T^{i,N}}\beta
(C_{N}(t,\mathbf{X}^{i,N}\left(  t\right)  ))g_{0}\int_{0}^{t}\left\vert
\mathbf{V}^{i,N}\left(  s\right)  \right\vert \mathbb{I}_{s\geq T^{i,N}}ds.
\]

The two branching processes introduced above can be represented as two
inhomogeneous Poisson processes with random rates, $N^{i,1}\left(  \int%
_{0}^{t}\lambda^{i,1}\left(  s\right)  ds\right)  ,$ and $N^{i,2}\left(
\int_{0}^{t}\lambda^{i,2}\left(  s\right)  ds\right)  ,$ for each particle
$i=1,\ldots,N_{t};$ here $N^{i,1}\left(  t\right)  $, $N^{i,2}\left(
t\right)  $, are standard Poisson processes of rate 1. Notice that all
processes in the family $\left\{  N^{i,1},N^{i,2},\mathbf{W}^{i}%
;i=1,\ldots,N_{t}\right\}  $ are independent.

When the process $N^{i,1}\left(  \int_{0}^{t}\lambda^{i,1}\left(  s\right)
ds\right)  $ jumps from 0 to 1 a new particle is created at $X^{i,N}\left(
t\right)  $; when the process $N^{i,2}\left(  \int_{0}^{t}\lambda^{i,2}\left(
s\right)  ds\right)  $ jumps from 0 to 1 a new particle is created along the
vessel $\left(  \mathbf{X}^{i,N}\left(  s\right)  \right)  _{s\in\left[
T^{i,N},t\right]  }$ (the position where it is created is assumed to be
uniformly distributed along the vessel, with respect to the relevant Hausdorff
measure). After each new creation there is a new tip with a new index, and its
own dynamics.

At the analytical level, we have the inequalities%
\begin{align*}
\lambda^{i,1}\left(  t\right)   &  \leq\mathbb{I}_{t\geq T^{i,N}}\left\Vert
\alpha\right\Vert _{\infty}g_{0};\\
\lambda^{i,2}\left(  t\right)   &  \leq\mathbb{I}_{t\geq T^{i,N}}\left\Vert
\beta\right\Vert _{\infty}g_{0}T\left(  C+CT+\sigma\sup_{s\in\left[
0,T\right]  }\left\vert \int_{0}^{s}e^{k_{1}r}d\mathbf{W}^{i}\left(  r\right)
\right\vert \right)  ,
\end{align*}
the second one due to the following lemma, used also below in other sections.
In the stochastic equation for $\mathbf{V}^{i,N}\left(  t\right)  $ it is not
restrictive to assume that the Brownian motion $\mathbf{W}^{i}\left(
t\right)  $ are defined for all $t\geq0$, not only for $t\geq T^{i,N}$.

\begin{lemma}
\label{lemma estimate V}There exists a constant $C>0$ such that
\[
\left\vert \mathbf{V}^{i,N}\left(  t\right)  \right\vert \leq e^{-k_{1}\left(
t-T^{i}\right)  }\left\vert \mathbf{V}_{0}^{i,N}\right\vert +\int_{T^{i}}%
^{t}Cds+\sigma\int_{0}^{t}e^{k_{1}s}d\mathbf{W}^{i}\left(  s\right)
\]
and also%
\[
\left\vert \mathbf{V}^{i,N}\left(  t\right)  \right\vert \leq C\left(
1+T\right)  +\sigma\int_{0}^{t}e^{k_{1}s}d\mathbf{W}^{i}\left(  s\right)  .
\]

\end{lemma}

\begin{proof}
From the variation of constant formula, we have%
\begin{align*}
\left\vert \mathbf{V}^{i,N}\left(  t\right)  \right\vert  &  \leq
e^{-k_{1}\left(  t-T^{i}\right)  }\left\vert \mathbf{V}_{0}^{i,N}\right\vert
\\
&  +\int_{T^{i}}^{t}e^{-k_{1}\left(  t-s\right)  }f\left(  \left\vert \nabla
C_{N}\left(  s,\mathbf{X}^{i,N}\left(  s\right)  \right)  \right\vert \right)
\left\vert \nabla C_{N}\left(  s,\mathbf{X}^{i,N}\left(  s\right)  \right)
\right\vert ds\\
&  +\sigma\int_{T^{i}}^{t}e^{-k_{1}\left(  t-s\right)  }d\mathbf{W}^{i}\left(
s\right)  .
\end{align*}
Then we use the bound from above for $f\left(  r\right)  r$ by a constant, see
(\ref{saturation}), and the boundedness of $\mathbf{V}_{0}^{i,N}$ (recall we
have assumed that their laws are compact support).
\end{proof}

We now introduce a new dominating process, without space structure, where the
times of birth of new particles are denoted by $\widetilde{T}^{i,N},$ for
$i=1,\ldots,\widetilde{N}_{t},$ having denoted by $\widetilde{N}_{t}$ the
total number of particles at time $t$ in this dominating process.

Given the same standard processes $N^{i,1},N^{i,2},W^{i}$ of the previous
process, take now as time-rates of branching
\begin{align*}
\widetilde{\lambda}^{i,1}\left(  t\right)   &  =\mathbb{I}_{t\geq
\widetilde{T}^{i,N}}\left\Vert \alpha\right\Vert _{\infty}g_{0}\\
\widetilde{\lambda}^{i,2}\left(  t\right)   &  =\mathbb{I}_{t\geq
\widetilde{T}^{i,N}}\left\Vert \beta\right\Vert _{\infty}g_{0}T\left(
C+CT+\sigma\sup_{s\in\left[  0,T\right]  }\left\vert \int_{0}^{s}e^{k_{1}%
r}d\mathbf{W}^{i}\left(  r\right)  \right\vert \right)  .
\end{align*}

We then consider the inhomogeneous Poisson processes $$N^{i,j}\left(  \int%
_{0}^{t}\widetilde{\lambda}^{i,j}\left(  s\right)  ds\right),
\,i=1,\ldots,\widetilde{N}_{t}, \, j=1,2 ;
$$ when they jump from 0 to 1, a new particle is created and the
system with the two new particles restart with the same rules.

Due to the path by path inequalities $\lambda^{i,j}\left(  t\right)
\leq\widetilde{\lambda}^{i,j}\left(  t\right)  $ and the fact that
$N^{i,1}\left(  t\right)  $, $N^{i,2}\left(  t\right)  $ are the same, the
times when the processes \newline$N^{i,j}\left(  \int_{0}^{t}\lambda
^{i,j}\left(  s\right)  ds\right)  $ jump from 0 to 1 are posterior to the
times when the processes $N^{i,j}\left(  \int_{0}^{t}\widetilde{\lambda}%
^{i,j}\left(  s\right)  ds\right)  $ jump from 0 to 1;\ precisely, this fact
is established in iterative manner, first on the particles that have
$T^{i,N}=\widetilde{T}^{i,N}=0$ (for which the inequalities $\lambda
^{i,j}\left(  t\right)  \leq\widetilde{\lambda}^{i,j}\left(  t\right)  $ are
directly true), then for the newborn particles, where $T^{i,N}\geq
\widetilde{T}^{i,N}$, hence $I_{t\geq T^{i,N}}\leq I_{t\geq\widetilde{T}%
^{i,N}}$ and thus again $\lambda^{i,j}\left(  t\right)  \leq\widetilde{\lambda
}^{i,j}\left(  t\right)  $.

The fact that the dominating process has the times of branching before the
original process implies that the total number of particles in the dominating
process is larger than in the original process, namely%
\[
N_{t}\leq\widetilde{N}_{t}.
\]
This is the result we wanted to obtain. Therefore, in order to have bounds
from above for $N_{t}$, it is sufficient to have them for $\widetilde{N}_{t}$.
Until now however we have solved only one of the difficulties posed by
branching along paths:\ we have dominated the space-dependent original process
by a much simpler one, without space structure. However, the dominating
process is not Yule, because the rate $\widetilde{\lambda}^{i,2}\left(
t\right)  $ is random, it depends on $W^{i}$. This dominating process, without
spatial structure, is of Cox-type, being made of inhomogeneous Poisson
processes with random rates of jump, but independent of the process itself.
Hence we are now faced with the second difficulty, namely estimating the
number of particles in this new process.

When we deal with the dominating process itself, without exploiting the
stochastic coupling with the original one, we may formalize it by saying that
we have random variables $Z^{i}$ distributed as%
\[
Z^{i}\overset{Law}{\sim}\left\Vert \alpha\right\Vert _{\infty}g_{0}%
+C\left\Vert \beta\right\Vert _{\infty}g_{0}T\left(  T+1\right)  +\left\Vert
\beta\right\Vert _{\infty}g_{0}T\sigma\sup_{s\in\left[  0,T\right]
}\left\vert \int_{0}^{s}e^{k_{1}r}d\mathbf{W}^{i}\left(  r\right)
\right\vert
\]
that are independent and equally distributed. Particle $i$ has a rate of
branching given by
\[
\widetilde{\lambda}^{i}\left(  t\right)  =\mathbb{I}_{t\geq\widetilde{T}%
^{i,N}}Z^{i}.
\]
We perform now a further reduction. The process we are considering starts with
$N$ particles. But due to its nature, completely non-interacting, it is the
same as $N$ independent copies of the same process starting from one particle.
Thus%
\[
\widetilde{N}_{t}=\sum_{k=1}^{N}\widetilde{N}_{t}^{\left(  k\right)  }%
\]
where, for each $k$, $\widetilde{N}_{t}^{\left(  k\right)  }$ is a process
like the dominating one, but with only one particle at the beginning; and the
processes with cardinality $\widetilde{N}_{t}^{\left(  k\right)  }$ are
independent and equally distributed. We have%
\begin{align*}
\mathbb{E}\left[  \left(  \frac{\widetilde{N}_{t}}{N}\right)  ^{p}\right]   &
=\mathbb{E}\left[  \left(  \frac{1}{N}\sum_{k=1}^{N}\widetilde{N}_{t}^{\left(
k\right)  }\right)  ^{p}\right]  \leq\mathbb{E}\left[  \frac{1}{N}\sum
_{k=1}^{N}\left(  \widetilde{N}_{t}^{\left(  k\right)  }\right)  ^{p}\right]
\\
&  =\frac{1}{N}\sum_{k=1}^{N}\mathbb{E}\left[  \left(  \widetilde{N}%
_{t}^{\left(  k\right)  }\right)  ^{p}\right]  =\mathbb{E}\left[  \left(
\widetilde{N}_{t}^{\left(  1\right)  }\right)  ^{p}\right]  .
\end{align*}
Hence, for the sake of $p$-moments of $\displaystyle\frac{\widetilde{N}_{t}%
}{N}$ (and a fortiori $\displaystyle\frac{N_{t}}{N}$), it is sufficient to
bound the $p$-moments of $\widetilde{N}_{t}^{\left(  1\right)  }$. A similar
fact holds for exponential moments, with a little more work. In the sequel, we
shall denote $\widetilde{N}_{t}^{\left(  1\right)  }$ by $\overline{N}_{t}$.

Let us analyze the dominating process $\overline{N}_{t}$ with analogous
formalism as the original space-dependent process; we do not need however to
index by $N$ all quantities, since this process starts with one particle only.
Let us denote by $\overline{T}^{i}$ the birth time of particle $i$, by
$\overline{Z}^{i}$ i.i.d. random variables as those above, and prescribe that
the first particle has index $i=1$, the second particle (the first newborn)
has index $i=2$, the third one $i=3$ and so on. Then branching is described by
a random measure $\overline{\Phi}$ on $\mathcal{B}_{\mathbb{R}^{+}}$ with
compensator given by
\[
\sum_{i=1}^{\overline{N}_{s}}\mathbb{I}_{s\geq\overline{T}^{i}}\overline
{Z}^{i}ds.
\]
Moreover, the random measure $\overline{\Phi}$ is given by $\overline{\Phi
}\left(  ds\right)  =\sum_{i=1}^{\overline{N}_{s}}\delta_{\overline{T}^{i}%
}\left(  ds\right)  $. Therefore $\overline{N}_{t}$ satisfies%
\[
\overline{N}_{t}=1+\int_{0}^{t}\sum_{i=1}^{\overline{N}_{s}}\mathbb{I}%
_{s\geq\overline{T}^{i}}\overline{Z}^{i}ds+\overline{M}_{t}%
\]
where $\overline{M}_{t}$ is the martingale
\[
\overline{M}_{t}=\int_{0}^{t}\overline{\Phi}\left(  ds\right)  -\int_{0}%
^{t}\sum_{i=1}^{\overline{N}_{s}}\mathbb{I}_{s\geq\overline{T}^{i}}%
\overline{Z}^{i}ds.
\]
We thus have%
\begin{align*}
\mathbb{E}\left[  \overline{N}_{t}\right]   &  =1+\int_{0}^{t}\mathbb{E}%
\left[  \sum_{i=1}^{\overline{N}_{s}}\mathbb{I}_{s\geq\overline{T}^{i}%
}\overline{Z}^{i}\right]  ds\\
&  \leq1+\int_{0}^{t}\mathbb{E}\left[  \sum_{i=1}^{\overline{N}_{s}}%
\overline{Z}^{i}\right]  ds.
\end{align*}
Wald's identity (proved below) tells us that $E\left[  \sum_{i=1}%
^{\overline{N}_{s}}\overline{Z}^{i}\right]  =E\left[  \overline{Z}^{1}\right]
E\left[  \overline{N}_{s}\right]  $; hence%
\[
\mathbb{E}\left[  \overline{N}_{t}\right]  \leq1+\int_{0}^{t}\mathbb{E}\left[
\overline{Z}^{1}\right]  \mathbb{E}\left[  \overline{N}_{s}\right]  ds
\]
which implies
\[
\mathbb{E}\left[  \overline{N}_{t}\right]  \leq e^{\mathbb{E}\left[
\overline{Z}^{1}\right]  t}.
\]
Since $\mathbb{E}\left[  \sup_{s\in\left[  0,T\right]  }\left\vert \int%
_{0}^{s}e^{k_{1}r}d\mathbf{W}^{1}\left(  r\right)  \right\vert \right]
<\infty$, we have completed the proof of the theorem, with $\lambda=E\left[
\overline{Z}^{1}\right]  $.

Let us explain the validity of Wald's identity. Notice that $\overline{N}_{t}$
increases from value $n$ to $n+1$ at time $\overline{T}_{n+1}$; this time
depends only on the first $n$ particles, hence on the r.v.'s $Z^{1},...,Z^{n}%
$. Hence $\overline{T}_{n+1}$ and $Z^{n+1}$ are independent. Therefore%
\begin{align*}
\mathbb{E}\left[  \sum_{n=1}^{\overline{N}_{t}}\overline{Z}^{n}\right]   &
=\sum_{k=1}^{\infty}\sum_{n=1}^{k}E\left[  \overline{Z}^{n}1_{\overline{N}%
_{t}=k}\right]  =\sum_{n=1}^{\infty}\sum_{k=n}^{\infty}E\left[  \overline
{Z}^{n}1_{\overline{N}_{t}=k}\right] \\
&  =\sum_{n=1}^{\infty}E\left[  \overline{Z}^{n}1_{\overline{N}_{t}\geq
n}\right]  =\sum_{n=1}^{\infty}E\left[  \overline{Z}^{n}1_{T_{n}\leq
t}\right]  =\sum_{n=1}^{\infty}E\left[  \overline{Z}^{n}\right]  P\left(
T_{n}\leq t\right) \\
&  =E\left[  \overline{Z}\right]  \sum_{n=1}^{\infty}P\left(  \overline{N}%
_{t}\geq n\right)  =E\left[  \overline{Z}\right]  E\left[  \overline{N}%
_{t}\right]  .
\end{align*}

\subsection{Final bounds on $C_{N}$\label{section final bounds on C}}

\begin{proposition}
For every $\epsilon>0$ there is $R>0$ such that
\begin{equation}
P\left(  \sup_{t\in\left[  0,T\right]  }\left\Vert \nabla C_{N}\left(
t\right)  \right\Vert _{E}>R\right)  \leq\epsilon\label{main bound 1}%
\end{equation}%
\begin{equation}
P\left(  \sup_{t\in\left[  0,T\right]  }\left\Vert D^{2}C_{N}\left(  t\right)
\right\Vert _{E}>R\right)  \leq\epsilon\label{main bound 2}%
\end{equation}
for all $N\in\mathbb{N}$.
\end{proposition}

\begin{proof}
Using Lemma \ref{lemma estimate V} in (\ref{est on C N}) we get
\[
\left\Vert \partial_{i}C_{N}\left(  t\right)  \right\Vert _{\infty}\leq
a_{0}+\int_{0}^{t}\frac{a_{1}}{\sqrt{t-s}}\int_{0}^{s}\frac{1}{N}\sum
_{i=1}^{N_{r}}1_{r\in\lbrack T^{i,N},\Theta^{i,N})}\left\vert \mathbf{V}%
^{i,N}\left(  r\right)  \right\vert drds
\]%
\begin{align*}
&  \leq a_{0}+\int_{0}^{t}\frac{a_{1}}{\sqrt{t-s}}\int_{0}^{s}\frac{N_{r}}%
{N}C\left(  1+T\right)  drds\\
&  +\int_{0}^{t}\frac{a_{1}}{\sqrt{t-s}}\int_{0}^{s}\frac{1}{N}\sum
_{i=1}^{N_{r}}\left(  \sigma\int_{0}^{r}e^{k_{1}u}d\mathbf{W}^{i}\left(
u\right)  \right)  drds
\end{align*}%
\[
\leq C+C\sup_{r\in\left[  0,T\right]  }\frac{N_{r}}{N}+C\frac{1}{N}\sum
_{i=1}^{\sup_{r\in\left[  0,T\right]  }N_{r}}\sup_{r\in\left[  0,T\right]
}\left\vert \int_{0}^{r}e^{k_{1}u}d\mathbf{W}^{i}\left(  u\right)  \right\vert
ds
\]
for some constant $C>0$. We apply the estimates and arguments of the previous
section (dominating $\sup_{r\in\left[  0,T\right]  }N_{r}$ from above as in
that section), including Wald identity for the second term, to get
\[
E\left[  \sup_{t\in\left[  0,T\right]  }\left\Vert \nabla C_{N}\left(
t\right)  \right\Vert _{E}\right]  \leq C
\]
and thus (\ref{main bound 1}). Using this bound and the same arguments, one
gets (\ref{main bound 2}) (here a bound in expected value is not known).
\end{proof}

\subsection{End of the proof\label{sect end of proof}}

Denote by $\mathcal{M}_{+}\left(  \mathbb{R}^{d}\times\mathbb{R}^{d}\right)  $
the space of finite positive Borel measures on $\mathbb{R}^{d}\times
\mathbb{R}^{d}$. Following \cite{StroockVarad} Chapter 1, see also
\cite{KipnisLandm} Chapter 4, weak convergence of measures in $\mathcal{M}%
\left(  \mathbb{R}^{d}\times\mathbb{R}^{d}\right)  $ is metrizable and a
metric is given by
\[
\delta\left(  \mu_{1},\mu_{2}\right)  =\sum_{k=1}^{\infty}2^{-k}\left(
\left\vert \mu_{1}\left(  \phi_{k}\right)  -\mu_{2}\left(  \phi_{k}\right)
\right\vert \wedge1\right)
\]
where $\left\{  \phi_{k}\right\}  $ is a suitable dense countable set in
$C_{b}\left(  \mathbb{R}^{d}\times\mathbb{R}^{d}\right)  $, and one can take
$\phi_{k}$ of class $UC_{b}^{1}\left(  \mathbb{R}^{d}\times\mathbb{R}%
^{d}\right)  $, the space of bounded uniformly continuous functions on
$\mathbb{R}^{d}\times\mathbb{R}^{d}$, with their first derivatives. Consider
the space $\mathcal{Y}:=C\left(  \left[  0,T\right]  ;\mathcal{M}_{+}\left(
\mathbb{R}^{d}\times\mathbb{R}^{d}\right)  \right)  $. Our first aim in this
section is to prove that the family of laws of $Q_{N}$, $N\in\mathbb{N}$, is
tight on $\mathcal{Y}$, namely for every $\epsilon>0$ there is a compact set
$K_{\epsilon}\subset\mathcal{Y}$ such that $P\left(  Q_{N}\in\mathcal{K}%
_{\epsilon}\right)  >1-\epsilon$. From Proposition 1.7 of \cite{KipnisLandm},
if we show that for every $k\in\mathbb{N}$ the family of laws on $C\left(
\left[  0,T\right]  \right)  $ of the real valued stochastic processes
$\left\langle Q_{N}\left(  t\right)  ,\phi_{k}\right\rangle $, $N\in
\mathbb{N}$, is tight, then the family of laws of $Q_{N}$, $N\in\mathbb{N}$,
is tight on $\mathcal{Y}$. For every $k\in\mathbb{N}$, thanks to Aldous
criterium (see \cite{KipnisLandm} Chapter 4), it is sufficient to prove two
conditions:\ for every $\epsilon>0$ there is $R>0$ such that
\begin{equation}
\mathbb{P}\left(  \left\vert \left\langle Q_{N}\left(  t\right)  ,\phi
_{k}\right\rangle \right\vert >R\right)  \leq\epsilon\label{tight 1}%
\end{equation}
for all $t\in\left[  0,T\right]  $ and $N\in\mathbb{N}$; and that for every
$\epsilon>0$
\begin{equation}
\lim_{\eta\rightarrow0}\underset{N\rightarrow\infty}{\lim\sup}\sup
_{\substack{\tau\in\Upsilon_{T}\\\theta\in\left[  0,\eta\right]  }%
}\mathbb{P}\left(  \left\vert \left\langle Q_{N}\left(  \tau+\theta\right)
,\phi_{k}\right\rangle -\left\langle Q_{N}\left(  \tau\right)  ,\phi
_{k}\right\rangle \right\vert >\epsilon\right)  =0\label{tight 2}%
\end{equation}
where $\Upsilon_{T}$ is the family of stopping times bounded by $T$.

\begin{proposition}
\label{Propos tight}Conditions (\ref{tight 1}) and (\ref{tight 2}) hold true.
\end{proposition}

\begin{proof}
\textbf{Step 1}. To prove the first condition, notice that
\[
\left\langle Q_{N}\left(  t\right)  ,\phi_{k}\right\rangle =\int%
_{\mathbb{R}^{d}\times\mathbb{R}^{d}}f_{k}\left(  \mathbf{x},\mathbf{v}%
\right)  Q_{N}\left(  t\right)  \left(  d\mathbf{x},d\mathbf{v}\right)
\leq\left\Vert \phi_{k}\right\Vert _{\infty}\frac{N_{t}}{N}%
\]
Hence we deduce (\ref{tight 1}) from Chebyshev inequality and Theorem
\ref{upper_bound}.

\textbf{Step 2}. To prove the second condition, notice that, from
the identity satisfied by $Q_{N}$,

\begin{eqnarray}
&& \left\langle Q_{N}\left( \tau+\theta\right)
,\phi_{k}\right\rangle -\left\langle Q_{N}\left(  \tau\right)
,\phi_{k}\right\rangle = \nonumber \\
&=&   \int_{\tau}^{\tau+\theta}\int_{\mathbb{R}^{d}\times\mathbb{R}^{d}%
}\mathbf{v}\cdot\nabla_{x}\phi_{k}\left(
\mathbf{x},\mathbf{v}\right) Q_{N}\left(  s\right)  \left(
d\mathbf{x},d\mathbf{v}\right)  ds  \nonumber \\
 &+& \int_{\tau}^{\tau+\theta}\int_{\mathbb{R}^{d}\times\mathbb{R}^{d}}\left[
f\left(  |\nabla C_{N}\left(  s,\mathbf{x}\right)  |\right) \nabla
C_{N}\left(  s,\mathbf{x}\right)  -k_{1}\mathbf{v}\right]  \times \nonumber \\
&& \qquad \qquad  \quad \times \nabla_{v}%
\phi_{k}\left(  \mathbf{x},\mathbf{v}\right)  Q_{N}\left( s\right)
\left( d\mathbf{x},d\mathbf{v}\right)  ds             \nonumber \\
&+& \int_{\tau}^{\tau+\theta}\int_{\mathbb{R}^{d}\times\mathbb{R}^{d}}%
\frac{\sigma^{2}}{2}\Delta_{v}\phi_{k}\left(
\mathbf{x},\mathbf{v}\right) Q_{N}\left(  s\right)  \left(
d\mathbf{x},d\mathbf{v}\right)  ds \nonumber \\
&+&
\int_{\tau}^{\tau+\theta}\int_{\mathbb{R}^{d}\times\mathbb{R}^{d}}\phi
_{G}^{k}\left(  \mathbf{x}\right)
\alpha(C_{N}(s,\mathbf{x}))Q_{N}\left( s\right)  \left(
d\mathbf{x},d\mathbf{v}\right)  ds  \nonumber \\
 &+& \int_{\tau}^{\tau+\theta}\int_{\mathbb{R}^{d}\times\mathbb{R}^{d}}\phi
_{G}^{k}\left(  \mathbf{x}\right)
\beta(C_{N}(s,\mathbf{x}))\left\vert \mathbf{v}\right\vert
\int_{0}^{s}Q_{N}\left(  r\right)  \left(
d\mathbf{x},d\mathbf{v}\right)  drds  \nonumber \\
&-& \gamma\int_{\tau}^{\tau+\theta}\int_{\mathbb{R}^{d}\times\mathbb{R}^{d}%
}\phi_{k}\left(  \mathbf{x},\mathbf{v}\right)  g\left(
s,\mathbf{x},\left\{ Q_{N}\left(  r\right)  \right\}  _{r\in\left[
0,s\right]  }\right) Q_{N}\left(  s\right)  \left(
d\mathbf{x},d\mathbf{v}\right)  ds  \nonumber \\
&+& \widetilde{M}_{N}^{k}\left(  \tau+\theta\right)  -\widetilde{M}_{N}%
^{k}\left(  \tau\right) \nonumber
\end{eqnarray}
where $\phi_{G}^{k}\left(  \mathbf{x}\right)  :=\int_{\mathbb{R}^{d}}%
G_{{v}_{0}}(\mathbf{v})\phi_{k}\left(  \mathbf{x},\mathbf{v}\right)
d\mathbf{v}$ and $\widetilde{M}_{N}^{k}\left(  t\right)  $ is the martingale
corresponding to the test function $\phi_{k}$. \ Then%
\[
\left\vert \left\langle Q_{N}\left(  \tau+\theta\right)  ,\phi_{k}%
\right\rangle -\left\langle Q_{N}\left(  \tau\right)  ,\phi_{k}\right\rangle
\right\vert
\]%
\begin{align*}
&  \leq\left\Vert \nabla_{x}\phi_{k}\right\Vert _{\infty}\int_{\tau}%
^{\tau+\theta}\int_{\mathbb{R}^{d}\times\mathbb{R}^{d}}\left\vert
\mathbf{v}\right\vert Q_{N}\left(  s\right)  \left(  d\mathbf{x}%
,d\mathbf{v}\right)  ds\\
&  +\left\Vert \nabla_{v}\phi_{k}\right\Vert _{\infty}\int_{\tau}^{\tau
+\theta}\int_{\mathbb{R}^{d}\times\mathbb{R}^{d}}\left[  C_{f}+k_{1}\left\vert
\mathbf{v}\right\vert \right]  Q_{N}\left(  s\right)  \left(  d\mathbf{x}%
,d\mathbf{v}\right)  ds\\
&  +\frac{\sigma^{2}}{2}\left\Vert \Delta_{v}\phi_{k}\right\Vert _{\infty}%
\int_{\tau}^{\tau+\theta}\int_{\mathbb{R}^{d}\times\mathbb{R}^{d}}Q_{N}\left(
s\right)  \left(  d\mathbf{x},d\mathbf{v}\right)  ds
\end{align*}%
\begin{align*}
&  +\left\Vert G_{{v}_{0}}\right\Vert _{1}\left\Vert \phi_{k}\right\Vert
_{\infty}\left\Vert \alpha\right\Vert _{\infty}\int_{\tau}^{\tau+\theta}%
\int_{\mathbb{R}^{d}\times\mathbb{R}^{d}}Q_{N}\left(  s\right)  \left(
d\mathbf{x},d\mathbf{v}\right)  ds\\
&  +\left\Vert G_{{v}_{0}}\right\Vert _{1}\left\Vert \phi_{k}\right\Vert
_{\infty}\left\Vert \beta\right\Vert _{\infty}\int_{\tau}^{\tau+\theta}%
\int_{\mathbb{R}^{d}\times\mathbb{R}^{d}}\left\vert \mathbf{v}\right\vert
\int_{0}^{s}Q_{N}\left(  r\right)  \left(  d\mathbf{x},d\mathbf{v}\right)
drds\\
&  +\gamma\left\Vert \phi_{k}\right\Vert _{\infty}\left\Vert g\right\Vert
_{\infty}\int_{\tau}^{\tau+\theta}\int_{\mathbb{R}^{d}\times\mathbb{R}^{d}%
}Q_{N}\left(  s\right)  \left(  d\mathbf{x},d\mathbf{v}\right)  ds\\
&  +\left\vert \widetilde{M}_{N}^{k}\left(  \tau+\theta\right)  -\widetilde{M}%
_{N}^{k}\left(  \tau\right)  \right\vert .
\end{align*}
Using this inequality, if we prove the validity of the limit (\ref{tight 2})
for each term of this sum, then we have proved (\ref{tight 2}). In the next
steps we shall analyze the various terms.

\textbf{Step 3}. Some of the terms above have the form ($C>0$ is a constant)%
\[
C\int_{\tau}^{\tau+\theta}\int_{\mathbb{R}^{d}\times\mathbb{R}^{d}}%
Q_{N}\left(  s\right)  \left(  d\mathbf{x},d\mathbf{v}\right)  ds=C\int_{\tau
}^{\tau+\theta}\frac{N_{s}}{N}ds\leq C\theta\sup_{s\in\left[  0,T\right]
}\frac{N_{s}}{N}%
\]
and therefore, for such terms,
\begin{align*}
&  \lim_{\varsigma\rightarrow0}\underset{N\rightarrow\infty}{\lim\sup}%
\sup_{\substack{\tau\in\Upsilon_{T}\\\theta\in\left[ 0,\,
\varsigma \, \right]
}}\mathbb{E}\left[  C\int_{\tau}^{\tau+\theta}\int_{\mathbb{R}^{d}%
\times\mathbb{R}^{d}}Q_{N}\left(  s\right)  \left(  d\mathbf{x},d\mathbf{v}%
\right)  ds\right]  \\
&  \leq C\lim_{\varsigma\rightarrow0}\underset{N\rightarrow\infty}{\lim\sup}%
\, \varsigma \, \mathbb{E}\left[  \sup_{s\in\left[  0,T\right]
}\frac{N_{s}}{N}\right] =0
\end{align*}
because $
\displaystyle{\underset{N\rightarrow\infty}{\lim\sup}\mathbb{E}\left[
\sup _{s\in\left[  0,T\right]  }\frac{N_{s}}{N}\right] } $ is
finite by Theorem \ref{upper_bound}. Therefore, for such terms, we
have (\ref{tight 2}) by Chebishev inequality.

\textbf{Step 4}. Other terms above have the form%
\begin{equation}
\small{
C\int_{\tau}^{\tau+\theta}\int_{\mathbb{R}^{d}\times\mathbb{R}^{d}}\left\vert
\mathbf{v}\right\vert Q_{N}\left(  s\right)  \left(  d\mathbf{x}%
,d\mathbf{v}\right)  ds=C\int_{\tau}^{\tau+\theta}\frac{1}{N}\sum_{i=1}%
^{N_{s}}\mathbb{I}_{s\in\lbrack T^{i,N},\Theta^{i,N})}\left\vert
\mathbf{V}^{i,N}\left(  t\right)  \right\vert ds.}
\end{equation}
From Lemma \ref{lemma estimate V} we have%
\begin{align*}
&  \leq C^{\prime}\int_{\tau}^{\tau+\theta}\left(  \frac{N_{s}}{N}+\frac{1}%
{N}\sum_{i=1}^{N_{s}}\sup_{r\in\left[  0,T\right]  }\left\vert \int_{0}%
^{r}e^{k_{1}u}d\mathbf{W}^{i}\left(  u\right)  \right\vert \right)  ds\\
&  =C^{\prime}\theta\sup_{s\in\left[  0,T\right]  }\left(  \frac{N_{s}}%
{N}+\frac{1}{N}\sum_{i=1}^{N_{s}}\sup_{r\in\left[  0,T\right]  }\left\vert
\int_{0}^{r}e^{k_{1}u}d\mathbf{W}^{i}\left(  u\right)  \right\vert \right)
\end{align*}
for a new constant $C^{\prime}>0$. We thus have%
\begin{align*}
&  \lim_{\varsigma\rightarrow0}\underset{N\rightarrow\infty}{\lim\sup}%
\sup_{\substack{\tau\in\Upsilon_{T}\\\theta\in\left[
0,\varsigma\right]
}}\mathbb{E}\left[  C\int_{\tau}^{\tau+\theta}\int_{\mathbb{R}^{d}%
\times\mathbb{R}^{d}}\left\vert \mathbf{v}\right\vert Q_{N}\left(  s\right)
\left(  d\mathbf{x},d\mathbf{v}\right)  ds\right]  \\
&  \leq
C^{\prime}\lim_{\varsigma\rightarrow0}\underset{N\rightarrow\infty
}{\lim\sup}\,\varsigma \,\mathbb{E}\left[  \sup_{s\in\left[
0,T\right] }\left(
\frac{N_{s}}{N}+\frac{1}{N}\sum_{i=1}^{N_{s}}\sup_{r\in\left[
0,T\right] }\left\vert \int_{0}^{r}e^{k_{1}u}d\mathbf{W}^{i}\left(
u\right)  \right\vert \right)  \right]  .
\end{align*}
The limit is zero concerning the first term, the one with $\frac{N_{s}}{N}$.
Let us discuss the second term. We have%
\[
\sup_{s\in\left[  0,T\right]  }\frac{1}{N}\sum_{i=1}^{N_{s}}\sup_{r\in\left[
0,T\right]  }\left\vert \int_{0}^{r}e^{k_{1}u}d\mathbf{W}^{i}\left(  u\right)
\right\vert =\frac{1}{N}\sum_{i=1}^{N_{T}^{\ast}}\sup_{r\in\left[  0,T\right]
}\left\vert \int_{0}^{r}e^{k_{1}u}d\mathbf{W}^{i}\left(  u\right)
\right\vert
\]
where $N_{T}^{\ast}=\sup_{s\in\left[  0,T\right]  }N_{s}$. Then we
apply the domination argument of Section \ref{section upper bound}
and Wald's identity, to
deduce%
\begin{equation}
\small{ \mathbb{E}\left[
\frac{1}{N}\sum_{i=1}^{N_{T}^{\ast}}\sup_{r\in\left[ 0,T\right]
}\left\vert \int_{0}^{r}e^{k_{1}u}d\mathbf{W}^{i}\left(  u\right)
\right\vert \right]  \leq\mathbb{E}\left[  \overline{N}_{T}\right]
\mathbb{E}\left[  \sup_{r\in\left[  0,T\right]  }\left\vert \int_{0}%
^{r}e^{k_{1}u}d\mathbf{W}^{i}\left(  u\right)  \right\vert \right]  \leq
C^{\prime\prime}}%
\end{equation}
for a constant $C^{\prime\prime}>0$, for every $N$. We thus deduce
\[
\lim_{\varsigma\rightarrow0}\underset{N\rightarrow\infty}{\lim\sup}\sup
_{\substack{\tau\in\Upsilon_{T}\\\theta\in\left[  0,\varsigma\right]  }%
}\mathbb{E}\left[  C\int_{\tau}^{\tau+\theta}\int_{\mathbb{R}^{d}%
\times\mathbb{R}^{d}}\left\vert \mathbf{v}\right\vert Q_{N}\left(  s\right)
\left(  d\mathbf{x},d\mathbf{v}\right)  ds\right]  =0
\]
and therefore we have (\ref{tight 2}) by Chebishev inequality, for the terms
just discussed. The proof for the term
\[
\int_{\tau}^{\tau+\theta}\int_{\mathbb{R}^{d}\times\mathbb{R}^{d}}\left\vert
\mathbf{v}\right\vert \int_{0}^{s}Q_{N}\left(  r\right)  \left(
d\mathbf{x},d\mathbf{v}\right)  drds
\]
is similar.

\textbf{Step 5}. Finally, we have to prove
\begin{equation}
\lim_{\varsigma\rightarrow0}\underset{N\rightarrow\infty}{\lim\sup}\sup
_{\substack{\tau\in\Upsilon_{T}\\\theta\in\left[  0,\varsigma\right]  }%
}\mathbb{P}\left(  \left\vert \widetilde{M}_{N}^{k}\left(  \tau+\theta\right)
-\widetilde{M}_{N}^{k}\left(  \tau\right)  \right\vert >\epsilon\right)
=0.\label{martingale limit}%
\end{equation}
We prove this separately for each one of the three martingales which compose
$\widetilde{M}_{N}^{k}$, that we call $\widetilde{M}_{i,N}^{k}$, $i=1,2,3$. We
follow a standard approach (see for instance \cite{KipnisLandm}). We use the
fact that
\begin{equation*}
\left(  \widetilde{M}_{1,N}^{k}\left(  t\right)  \right)  ^{2}-\int_{0}%
^{t}\frac{1}{N^{2}}\sum_{i=1}^{N_{s}}\left\vert \nabla_{v}\phi_{k}%
(X^{i,N}(s),V^{i,N}(s))\right\vert ^{2}I_{s\in\lbrack T^{i,N},\Theta^{i,N})}ds
\end{equation*}

\small{
\begin{eqnarray*}
 && \left(  \widetilde{M}_{2,N}^{k}\left(  t\right)  \right)  ^{2}
 \nonumber \\
&-&\int_{0}%
^{t}\int_{\mathbb{R}^{d}\times\mathbb{R}^{d}}\int_{\mathbb{R}^{d}}G_{{v}_{0}%
}(\mathbf{v})\phi_{k}^{2}\left(
\mathbf{x},\mathbf{v}^{\prime}\right)
d\mathbf{v}^{\prime}\alpha(C_{N}(s,\mathbf{x}))Q_{N}\left(
s\right)  \left( d\mathbf{x},d\mathbf{v}\right)  ds  \nonumber \\
 &-&  \int_{0}^{t}\int_{\mathbb{R}^{d}\times\mathbb{R}^{d}}\int_{\mathbb{R}^{d}%
}G_{{v}_{0}}(\mathbf{v})\phi_{k}^{2}\left(
\mathbf{x},\mathbf{v}^{\prime }\right)
d\mathbf{v}^{\prime}\beta(C_{N}(s,\mathbf{x}))\left\vert
\mathbf{v}\right\vert \int_{0}^{s}Q_{N}\left(  r\right)  \left(
d\mathbf{x},d\mathbf{v}\right)  drds
\end{eqnarray*}
}

\small{
\begin{equation*}
\left(  \widetilde{M}_{3,N}^{k}\left(  t\right)  \right)  ^{2}-\int_{0}%
^{t}\int_{\mathbb{R}^{d}\times\mathbb{R}^{d}}\int_{\mathbb{R}^{d}}\phi_{k}%
^{2}\left(  \mathbf{x},\mathbf{v}\right)  \gamma g\left(  s,\mathbf{x}%
,\left\{  Q_{N}\left(  r\right)  \right\}  _{r\in\left[ 0,s\right]
}\right) Q_{N}\left(  s\right)  \left(
d\mathbf{x},d\mathbf{v}\right)  ds
\end{equation*}
}
are martingales. One has%
\[
\mathbb{E}\left[  \left\vert \widetilde{M}_{1,N}^{k}\left(  \tau
+\theta\right)  -\widetilde{M}_{1,N}^{k}\left(  \tau\right)  \right\vert
^{2}\right]  =\mathbb{E}\left[  \left\vert \widetilde{M}_{1,N}^{k}\left(
\tau+\theta\right)  \right\vert ^{2}\right]  -E\left[  \left\vert
\widetilde{M}_{1,N}^{k}\left(  \tau\right)  \right\vert ^{2}\right]
\]%
\begin{align*}
&  =\mathbb{E}\left[  \int_{\tau}^{\tau+\theta}\frac{1}{N^{2}}\sum
_{i=1}^{N_{s}}\left\vert \nabla_{v}\phi_{k}(X^{i,N}(s),V^{i,N}(s))\right\vert
^{2}I_{s\in\lbrack T^{i,N},\Theta^{i,N})}ds\right]  \\
&  \leq\theta\left\Vert \nabla_{v}\phi_{k}\right\Vert _{\infty}^{2}%
\mathbb{E}\left[  \sup_{s\in\left[  0,T\right]  }\frac{N_{s}}{N^{2}}\right]
\end{align*}
and this implies (\ref{martingale limit}) for $\widetilde{M}_{1,N}^{k}$.
Similarly%
\begin{align*}
&  \mathbb{E}\left[  \left\vert \widetilde{M}_{2,N}^{k}\left(  \tau
+\theta\right)  -\widetilde{M}_{2,N}^{k}\left(  \tau\right)  \right\vert
^{2}\right]  \\
&  =\mathbb{E}\left[  \int_{\tau}^{\tau+\theta}\int_{\mathbb{R}^{d}%
\times\mathbb{R}^{d}}\int_{\mathbb{R}^{d}}G_{{v}_{0}}(\mathbf{v})\phi_{k}%
^{2}\left(  \mathbf{x},\mathbf{v}^{\prime}\right)  d\mathbf{v}^{\prime}%
\alpha(C_{N}(s,\mathbf{x}))Q_{N}\left(  s\right)  \left(  d\mathbf{x}%
,d\mathbf{v}\right)  ds\right]  \\
&  +\mathbb{E}\left[  \int_{\tau}^{\tau+\theta}\int_{\mathbb{R}^{d}%
\times\mathbb{R}^{d}}\int_{\mathbb{R}^{d}}G_{{v}_{0}}(\mathbf{v})\phi_{k}%
^{2}\left(  \mathbf{x},\mathbf{v}^{\prime}\right)  d\mathbf{v}^{\prime}%
\beta(C_{N}(s,\mathbf{x}))\left\vert \mathbf{v}\right\vert \int_{0}^{s}%
Q_{N}\left(  r\right)  \left(  d\mathbf{x},d\mathbf{v}\right)  drds\right]  .
\end{align*}
We bound these terms as above by $C\theta$; we do not repeat the computations.
Using the fact that $g$ is bounded, the proof for $\widetilde{M}_{3,N}^{k}$ is
similar. The proof of the proposition is complete.
\end{proof}

We have proved that the family of laws of $Q_{N}$, $N\in\mathbb{N}$, is tight
on $\mathcal{Y}$. The tightness of the sequence of laws of $C_{N}$,
$N\in\mathbb{N}$, on $C\left(  \left[  0,T\right]  ;UC_{b}^{1}\left(
\mathbb{R}^{d}\right)  \right)  $ can be proved using (\ref{main bound 1}%
)-(\ref{main bound 2}) and a few classical additional PDE arguments for the
compactness in time. We omit the details. Therefore the family of laws of the
pair $\left(  Q_{N},C_{N}\right)  $, $N\in\mathbb{N}$, is tight on
$\mathcal{Y}\times C\left(  \left[  0,T\right]  ;UC_{b}^{1}\left(
\mathbb{R}^{d}\right)  \right)  $. By Prohorov theorem, there exist weakly
convergent subsequences. A uniform in $N$ bound in expectation on $Q_{N}$ in
$L^{\infty}\left(  0,T;\mathcal{M}_{1}\left(  \mathbb{R}^{d}\times
\mathbb{R}^{d}\right)  \right)  $ implies that that the same bound holds for
any limit point of $\left(  Q_{N},C_{N}\right)  $.

The proof that the limit is supported on solutions of the limit system is
again classical (see e.g. \cite{KipnisLandm}, Chapter 4). The conclusion of
the proof of Theorem \ref{thm convergence} has been outlined at the beginning
of the section.

\section{Proof of Theorem \ref{thm uniqueness} \label{sct uniqueness}}

In this section we prove Theorem \ref{thm uniqueness}. The existence claim of
this theorem follows from the tightness and passage to the limit result proved
above, with the limit taken along any converging subsequence. Here we prove
uniqueness; it provides the convergence of the full sequence above.

Let us recall that we are studying the PDE system
\begin{align}
&  \partial_{t}p_{t}+\mathbf{v}\cdot\nabla_{x}p_{t}+\operatorname{div}%
_{v}\left(  \left[  F\left(  C_{t}\right)  -k_{1}\mathbf{v}\right]
p_{t}\right) \nonumber\\
&  =\frac{\sigma^{2}}{2}\Delta_{v}p_{t}+\Lambda_{t}\left(  C_{t}%
,p_{t},\left\{  \widetilde{p}_{r}\right\}  _{r\in\left[  0,t\right]  }\right)
\left(  d\mathbf{x},d\mathbf{v}\right)  , \label{PDE for p}%
\end{align}%
\begin{equation}
\partial_{t}C_{t}=k_{2}\delta_{A}+d_{1}\Delta C_{t}-\varsigma\left(  t,\mathbf{x}%
,\left\{  p_{r}\right\}  _{r\in\left[  0,t\right]  }\right)  C_{t}
\label{PDE for C}%
\end{equation}
subject to initial conditions $p_{0}\in\mathcal{M}_{1}\left(  \mathbb{R}%
^{d}\times\mathbb{R}^{d}\right)  $ and $C_{0}\in C_{b}^{1}\left(
\mathbb{R}^{d}\right)  $. We study this system in the case when $C_{t}%
=C\left(  t,\mathbf{x}\right)  $ is a regular function, of class $C\left(
\left[  0,T\right]  ;C_{b}^{1}\left(  \mathbb{R}^{d}\right)  \right)  $ but
$p_{t}=p_{t}\left(  d\mathbf{x},d\mathbf{v}\right)  $ is only a time-dependent
finite measure. The meaning of solution for (\ref{PDE for C}) is the mild
sense%
\[
C_{t}=e^{tA}C_{0}+\int_{0}^{t}e^{\left(  t-s\right)  A}\left(
k_{2}\delta _{A}-\varsigma\left(  s,\cdot,\left\{  p_{r}\right\}
_{r\in\left[  0,s\right] }\right)  C_{s}\right)  ds
\]
(see Section \ref{sect assumptions} for details). The meaning of solution for
(\ref{PDE for p}) is the weak sense: an element $p_{t}\left(  d\mathbf{x}%
,d\mathbf{v}\right)  $ of $L^{\infty}\left(  0,T;\mathcal{M}_{1}\left(
\mathbb{R}^{d}\times\mathbb{R}^{d}\right)  \right)  $ will be called a
measure-valued solution of the PDE (\ref{PDE for p}) if it satisfies the
identity (\ref{PDE weak form}) for all $\phi\in C_{c}^{\infty}\left(
\mathbb{R}^{d}\times\mathbb{R}^{d}\right)  $.

Let us recall or explain several notations. Given a measure $p_{t}\left(
d\mathbf{x},d\mathbf{v}\right)  $ in $L^{\infty}\left(  0,T;\mathcal{M}%
_{1}\left(  \mathbb{R}^{d}\times\mathbb{R}^{d}\right)  \right)  ,$ as above we
denote by $\tilde{p}_{t}=\tilde{p}_{t}(dx)$ and $\pi_{1}p_{t}=\left(  \pi
_{1}p_{t}\right)  \left(  d\mathbf{x}\right)  $ the measures defined as
\begin{align*}
\tilde{p}_{t}(d\mathbf{x})  &  :=\int_{\mathbb{R}^{d}}\left\vert
\mathbf{v}\right\vert p_{t}\left(  d\mathbf{x},d\mathbf{v}\right)  ,\\
\left(  \pi_{1}p_{t}\right)  \left(  d\mathbf{x}\right)   &  :=\int%
_{\mathbb{R}^{d}}p_{t}\left(  d\mathbf{x},d\mathbf{v}\right)  .
\end{align*}
More formally, on test functions $\phi\in C_{c}^{\infty}\left(  \mathbb{R}%
^{d}\right)  $,
\[
\int_{\mathbb{R}^{d}}\phi\left(  \mathbf{x}\right)  \tilde{p}_{t}%
(d\mathbf{x}):=\int_{\mathbb{R}^{d}\times\mathbb{R}^{d}}\phi\left(
\mathbf{x}\right)  \left\vert \mathbf{v}\right\vert p_{t}\left(
d\mathbf{x},d\mathbf{v}\right)  ,
\]%
\[
\int_{\mathbb{R}^{d}}\phi\left(  \mathbf{x}\right)  \left(  \pi_{1}%
p_{t}\right)  \left(  d\mathbf{x}\right)  :=\int_{\mathbb{R}^{d}%
\times\mathbb{R}^{d}}\phi\left(  \mathbf{x}\right)  p_{t}\left(
d\mathbf{x},d\mathbf{v}\right)  .
\]
Moreover, we will denote by $\Lambda_{t}\left(  C_{t},p_{t},\left\{
\widetilde{p}_{r}\right\}  _{r\in\left[  0,t\right]  }\right)  \left(
d\mathbf{x},d\mathbf{v}\right)  $ the measure defined on test functions
$\psi\in C_{c}^{\infty}\left(  \mathbb{R}^{d}\times\mathbb{R}^{d}\right)  $ as%
\begin{align}
&  \int_{\mathbb{R}^{d}\times\mathbb{R}^{d}}\psi\left(  \mathbf{x}%
,\mathbf{v}\right)  \Lambda_{t}\left(  C_{t},p_{t},\left\{  \widetilde{p}%
_{r}\right\}  _{r\in\left[  0,t\right]  }\right)  \left(  d\mathbf{x}%
,d\mathbf{v}\right) \nonumber\\
&  =\int_{\mathbb{R}^{d}\times\mathbb{R}^{d}}\psi\left(  \mathbf{x}%
,\mathbf{v}\right)  G_{{v}_{0}}\left(  \mathbf{v}\right)  \left(  \alpha
(C_{t}\left(  \mathbf{x}\right)  )\left(  \pi_{1}p_{t}\right)  \left(
d\mathbf{x}\right)  +\beta(C_{t}\left(  \mathbf{x}\right)  )\int_{0}%
^{t}\widetilde{p}_{r}\left(  d\mathbf{x}\right)  dr\right)  d\mathbf{v}%
\nonumber\\
&  -\gamma\int_{\mathbb{R}^{d}\times\mathbb{R}^{d}}\psi\left(  \mathbf{x}%
,\mathbf{v}\right)  h\left(  \int_{0}^{t}\left(  K_{2}\ast\widetilde{p}%
_{r}\right)  \left(  \mathbf{x}\right)  dr\right)  p_{t}\left(  d\mathbf{x}%
,d\mathbf{v}\right)
\end{align}
where $\left(  K_{2}\ast\widetilde{p}_{r}\right)  \left(  \mathbf{x}\right)  $
is the function defined as%
\[
\left(  K_{2}\ast\widetilde{p}_{r}\right)  \left(  \mathbf{x}\right)
=\int_{\mathbb{R}^{d}}K_{2}\left(  \mathbf{x}-\mathbf{x}^{\prime}\right)
\widetilde{p}_{r}(d\mathbf{x}^{\prime}).
\]
Finally,
\begin{align*}
F\left(  C_{t}\right)   &  :=f\left(  \left\vert \nabla C_{t}\right\vert
\right)  \nabla C_{t}\\
\varsigma\left(  t,\mathbf{x},\left\{  p_{r}\right\}  _{r\in\left[
0,t\right]
}\right)   &  :=\int_{0}^{t}\left(  \int_{\mathbb{R}^{d}\times\mathbb{R}^{d}%
}K_{1}\left(  \mathbf{x}-\mathbf{x}^{\prime}\right)  \left\vert \mathbf{v}%
^{\prime}\right\vert p_{r}\left(  d\mathbf{x}^{\prime},d\mathbf{v}^{\prime
}\right)  \right)  dr.
\end{align*}
We may shorten the notations and set
\[
\Lambda_{t}\left(  C_{t},p_{t},\widetilde{p}_{\cdot}\right)  :=\Lambda
_{t}\left(  C_{t},p_{t},\left\{  \widetilde{p}_{r}\right\}  _{r\in\left[
0,t\right]  }\right)  \left(  d\mathbf{x},d\mathbf{v}\right)  ;
\]

{\small
\[
\eta_{t}\left(  p_{\cdot}\right)  :=\eta\left(  t,\mathbf{x},\left\{
p_{r}\right\}  _{r\in\left[  0,t\right]  }\right)  .
\]
}

In the sequel we denote by $\left\langle \mu,\phi\right\rangle $ the integral%
\[
\left\langle \mu,\phi\right\rangle =\int_{\mathbb{R}^{2d}}\phi\left(
\mathbf{x},\mathbf{v}\right)  \mu\left(  d\mathbf{x},d\mathbf{v}\right)
\]
when $\mu\in\mathcal{M}_{1}\left(  \mathbb{R}^{d}\times\mathbb{R}^{d}\right)
$ and $\phi$ is such that this integral is well defined.

Let $\left(  p^{\prime},C^{\prime}\right)  $ and $\left(  p^{\prime\prime
},C^{\prime\prime}\right)  $ be two solutions with the regularity required in
the statement of the theorem. We use the distance
\[
d\left(  \mu^{\prime},\mu^{\prime\prime}\right)  :=\sup_{\left\Vert
\phi\right\Vert _{\infty}\leq1}\left\vert \left\langle \mu_{t}^{1}-\mu_{t}%
^{2},\left(  1+\left\vert \mathbf{v}\right\vert \right)  \phi\right\rangle
\right\vert
\]
on $\mathcal{M}_{1}\left(  \mathbb{R}^{d}\times\mathbb{R}^{d}\right)  $, where
the supremum is taken over all measurable bounded functions $\phi
:\mathbb{R}^{d}\times\mathbb{R}^{d}\rightarrow\mathbb{R}$ with $\left\Vert
\phi\right\Vert _{\infty}\leq1$. Below, we repeatedly use the inequality%
\[
\left\vert \left\langle \mu_{t}^{1}-\mu_{t}^{2},\left(  1+\left\vert
\mathbf{v}\right\vert \right)  \phi\right\rangle \right\vert \leq d\left(
\mu^{\prime},\mu^{\prime\prime}\right)  \left\Vert \phi\right\Vert _{\infty}%
\]
which holds true for all bounded measurable functions $\phi$; and therefore%
\[
\left\vert \left\langle \mu_{t}^{1}-\mu_{t}^{2},\varphi\psi\right\rangle
\right\vert \leq d\left(  \mu^{\prime},\mu^{\prime\prime}\right)  \left\Vert
\varphi\right\Vert _{\infty}\left\Vert \frac{1}{1+\left\vert \mathbf{v}%
\right\vert }\psi\right\Vert _{\infty}%
\]
for all bounded measurable functions $\varphi$ and all measurable functions
$\psi$ such that $\displaystyle\frac{1}{1+\left\vert \mathbf{v}\right\vert
}\psi$ is bounded.

Let us introduce the operator $Lf=\displaystyle\frac{\sigma^{2}}{2}\Delta
_{v}f-v\cdot\nabla_{x}f-k_{1}\operatorname{div}_{v}\left(  \mathbf{v}f\right)
$ over all smooth functions $f:\mathbb{R}^{d}\times\mathbb{R}^{d}%
\rightarrow\mathbb{R}$. We denote by $L^{\ast}$ its dual operator:%
\[
L^{\ast}\phi=\frac{\sigma^{2}}{2}\Delta_{v}\phi+\mathbf{v}\cdot\nabla_{x}%
\phi+k_{1}\mathbf{v}\cdot\nabla_{v}\phi,
\]
and by $e^{tL^{\ast}}$ its associated semigroup. Formally, if $p_{t}$ is a
solution of the equation above, then we have ($e^{tL}$ denotes the semigroup
associated with $L$)%
\[
p_{t}=e^{tL}p_{0}+\int_{0}^{t}e^{\left(  t-s\right)  L}\left(  \Lambda
_{s}\left(  C_{s},p_{s},\tilde{p}_{\cdot}\right)  -\operatorname{div}%
_{v}\left(  F\left(  C_{s}\right)  p_{s}\right)  \right)  ds
\]%
\[
\left\langle p_{t},\phi\right\rangle =\left\langle p_{0},e^{tL^{\ast}}%
\phi\right\rangle +\int_{0}^{t}\left\langle \Lambda_{s}\left(  C_{s}%
,p_{s},\tilde{p}_{\cdot}\right)  -\operatorname{div}_{v}\left(  F\left(
C_{s}\right)  p_{s}\right)  ,e^{\left(  t-s\right)  L^{\ast}}\phi\right\rangle
ds
\]
and therefore%
\begin{align}
\left\langle p_{t},\phi\right\rangle =\left\langle p_{0},e^{tL^{\ast}}%
\phi\right\rangle  &  +\int_{0}^{t}\left\langle \Lambda_{s}\left(  C_{s}%
,p_{s},\tilde{p}_{\cdot}\right)  ,e^{\left(  t-s\right)  L^{\ast}}%
\phi\right\rangle ds\nonumber\\
&  +\int_{0}^{t}\left\langle p_{s},F\left(  C_{s}\right)  \cdot\nabla
_{v}e^{\left(  t-s\right)  L^{\ast}}\phi\right\rangle ds.
\end{align}
This identity can be rigorously proved from the weak formulation of the
equation for $p$, first extending it to time-dependent test functions and then
by taking the test function $e^{tL^{\ast}}\phi$; we omit the lengthy but not
difficult computations.

From the previous identity we estimate%
\begin{align}
&  \left\vert \left\langle p_{t}^{\prime}-p_{t}^{\prime\prime},\left(
1+\left\vert \mathbf{v}\right\vert \right)  \phi\right\rangle \right\vert
\nonumber\\
&  \quad\leq\int_{0}^{t}\left\langle \Lambda_{s}\left(  C_{s}^{\prime}%
,p_{s}^{\prime},\tilde{p}_{\cdot}^{\prime}\right)  -\Lambda_{s}\left(
C_{s}^{\prime\prime},p_{s}^{\prime\prime},\tilde{p}_{\cdot}^{\prime\prime
}\right)  ,e^{\left(  t-s\right)  L^{\ast}}\left(  1+\left\vert \mathbf{v}%
\right\vert \right)  \phi\right\rangle ds\nonumber\\
&  \quad+\int_{0}^{t}\left\vert \left\langle p_{s}^{\prime},F\left(
C_{s}^{\prime}\right)  \cdot\nabla_{v}e^{\left(  t-s\right)  L^{\ast}}\left(
1+\left\vert \mathbf{v}\right\vert \right)  \phi\right\rangle \right.
\nonumber\\
&  \qquad\qquad\quad-\left.  \left\langle p_{s}^{\prime\prime},F\left(
C_{s}^{\prime\prime}\right)  \cdot\nabla_{v}e^{\left(  t-s\right)  L^{\ast}%
}\left(  1+\left\vert \mathbf{v}\right\vert \right)  \phi\right\rangle
\right\vert ds.
\end{align}
Hence, with the notation $\widetilde{\phi}=\left(  1+\left\vert \mathbf{v}%
\right\vert \right)  \phi$,
\[
\left\vert \left\langle p_{t}^{\prime}-p_{t}^{\prime\prime},\phi\right\rangle
\right\vert \leq\int_{0}^{t}\left(  I_{s,t}^{1}+I_{s,t}^{2}+I_{s,t}%
^{3}\right)  ds
\]%
\begin{align*}
I_{s,t}^{1}  &  =\left\vert \left\langle \Lambda_{s}\left(  C_{s}^{\prime
},p_{s}^{\prime},\tilde{p}_{\cdot}^{\prime}\right)  -\Lambda_{s}\left(
C_{s}^{\prime\prime},p_{s}^{\prime\prime},\tilde{p}_{\cdot}^{\prime\prime
}\right)  ,e^{\left(  t-s\right)  L^{\ast}}\widetilde{\phi}\right\rangle
\right\vert \\
I_{s,t}^{2}  &  =\left\vert \left\langle p_{s}^{\prime}-p_{s}^{\prime\prime
},F\left(  C_{s}^{\prime}\right)  \cdot\nabla_{v}e^{\left(  t-s\right)
L^{\ast}}\widetilde{\phi}\right\rangle \right\vert \\
I_{s,t}^{3}  &  =\left\vert \left\langle p_{s}^{\prime\prime},\left(  F\left(
C_{s}^{\prime}\right)  -F\left(  C_{s}^{\prime\prime}\right)  \right)
\cdot\nabla_{v}e^{\left(  t-s\right)  L^{\ast}}\widetilde{\phi}\right\rangle
\right\vert .
\end{align*}
Now we estimate%
\[
I_{s,t}^{1}\leq I_{s,t}^{1,1}+I_{s,t}^{1,2}+I_{s,t}^{1,3}%
\]%
\begin{align*}
&  I_{s,t}^{1,1}=\left\vert \left\langle \pi_{1}p_{s}^{\prime}\otimes
d\mathcal{L}_{v},G_{\mathbf{v}_{0}}\alpha(C_{s}^{\prime})e^{\left(
t-s\right)  L^{\ast}}\widetilde{\phi}\right\rangle \right. \\
&  \qquad\qquad-\left.  \left\langle \pi_{1}p_{s}^{\prime\prime}\otimes
d\mathcal{L}_{v},G_{\mathbf{v}_{0}}\alpha(C_{s}^{\prime\prime})e^{\left(
t-s\right)  L^{\ast}}\widetilde{\phi}\right\rangle \right\vert \\
&  \quad\leq\left\vert \left\langle \left(  \pi_{1}p_{s}^{\prime}-\pi_{1}%
p_{s}^{\prime\prime}\right)  \otimes d\mathcal{L}_{v},G_{\mathbf{v}_{0}}%
\alpha(C_{s}^{\prime})e^{\left(  t-s\right)  L^{\ast}}\widetilde{\phi
}\right\rangle \right\vert \\
&  \qquad\quad+\left\vert \left\langle \pi_{1}p_{s}^{\prime\prime}\otimes
d\mathcal{L}_{v},G_{\mathbf{v}_{0}}\left(  \alpha(C_{s}^{\prime})-\alpha
(C_{s}^{\prime\prime})\right)  e^{\left(  t-s\right)  L^{\ast}}\widetilde{\phi
}\right\rangle \right\vert \\
&  \quad\leq\left\Vert G_{\mathbf{v}_{0}}\left(  1+\left\vert \mathbf{v}%
\right\vert \right)  \right\Vert _{L^{1}}\left\Vert \alpha\right\Vert
_{\infty}\left\Vert \frac{1}{1+\left\vert \mathbf{v}\right\vert }e^{\left(
t-s\right)  L^{\ast}}\widetilde{\phi}\right\Vert _{\infty}d\left(
p_{s}^{\prime},p_{s}^{\prime\prime}\right) \\
&  \quad+\left\langle \pi_{1}p_{s}^{\prime\prime},1\right\rangle \left\Vert
G_{\mathbf{v}_{0}}\left(  1+\left\vert \mathbf{v}\right\vert \right)
\right\Vert _{L^{1}}\left\Vert \alpha^{\prime}\right\Vert _{\infty}\left\Vert
\frac{1}{1+\left\vert \mathbf{v}\right\vert }e^{\left(  t-s\right)  L^{\ast}%
}\widetilde{\phi}\right\Vert _{\infty}\left\Vert C_{s}^{\prime}-C_{s}%
^{\prime\prime}\right\Vert _{\infty};
\end{align*}

{\small
\begin{align*}
&  I_{s,t}^{1,2}=\int_{0}^{s}\left\vert \left\langle \tilde{p}_{r}^{\prime
}\otimes d\mathcal{L}_{v},G_{\mathbf{v}_{0}}\beta(C_{s}^{\prime})e^{\left(
t-s\right)  L^{\ast}}\widetilde{\phi}\right\rangle \right. \\
&  \quad\qquad-\left.  \left\langle \tilde{p}_{r}^{\prime\prime}\otimes
d\mathcal{L}_{v},G_{\mathbf{v}_{0}}\beta(C_{s}^{\prime\prime})e^{\left(
t-s\right)  L^{\ast}}\widetilde{\phi}\right\rangle \right\vert dr\\
&  \quad\leq\int_{0}^{s}\left\vert \left\langle \left(  \tilde{p}_{r}^{\prime
}-\tilde{p}_{r}^{\prime\prime}\right)  \otimes d\mathcal{L}_{v},G_{\mathbf{v}%
_{0}}\beta(C_{s}^{\prime})e^{\left(  t-s\right)  L^{\ast}}\widetilde{\phi
}\right\rangle \right\vert dr\\
&  \qquad+\int_{0}^{s}\left\vert \left\langle \tilde{p}_{r}^{\prime\prime
}\otimes d\mathcal{L}_{v},G_{\mathbf{v}_{0}}\left(  \beta(C_{s}^{\prime
})-\beta(C_{s}^{\prime\prime})\right)  e^{\left(  t-s\right)  L^{\ast}%
}\widetilde{\phi}\right\rangle \right\vert dr\\
&  \quad\leq\left\Vert G_{\mathbf{v}_{0}}\left(  1+\left\vert \mathbf{v}%
\right\vert \right)  \right\Vert _{L^{1}}\left\Vert \beta\right\Vert _{\infty
}\left\Vert \frac{1}{1+\left\vert \mathbf{v}\right\vert }e^{\left(
t-s\right)  L^{\ast}}\widetilde{\phi}\right\Vert _{\infty}\int_{0}^{s}d\left(
p_{r}^{\prime},p_{r}^{\prime\prime}\right)  dr\\
&  \,\,+\left\Vert G_{\mathbf{v}_{0}}\left(  1+\left\vert \mathbf{v}%
\right\vert \right)  \right\Vert _{L^{1}}\left\Vert \beta^{\prime}\right\Vert
_{\infty}\left\Vert \frac{1}{1+\left\vert \mathbf{v}\right\vert }e^{\left(
t-s\right)  L^{\ast}}\widetilde{\phi}\right\Vert _{\infty}\left(  \int_{0}%
^{s}\left\langle \tilde{p}_{r}^{\prime\prime},1\right\rangle dr\right)
\left\Vert C_{s}^{\prime}-C_{s}^{\prime\prime}\right\Vert _{\infty};
\end{align*}
}%

\begin{align*}
I_{s,t}^{1,3}  &  =\gamma\left\vert \left\langle
p_{s}^{\prime},h\left( \int_{0}^{s}\left(
K_{2}\ast\widetilde{p}_{r}^{\prime}\right)  dr\right) e^{\left(
t-s\right)  L^{\ast}}\widetilde{\phi}\right\rangle \right. \\
&  \qquad \qquad - \left. \left\langle
p_{s}^{\prime\prime},h\left(  \int_{0}^{s}\left(  K_{2}\ast\widetilde{p}%
_{r}^{\prime\prime}\right)  dr\right)  e^{\left(  t-s\right)  L^{\ast}%
}\widetilde{\phi}\right\rangle \right\vert \\
&  \leq\gamma\left\vert \left\langle p_{s}^{\prime}-p_{s}^{\prime\prime
},h\left(  \int_{0}^{s}\left(  K_{2}\ast\widetilde{p}_{r}^{\prime}\right)
dr\right)  e^{\left(  t-s\right)  L^{\ast}}\widetilde{\phi}\right\rangle
\right\vert \\
&  +\gamma\left\Vert h^{\prime}\right\Vert _{\infty}\left\vert \left\langle
p_{s}^{\prime\prime},\int_{0}^{s}K_{2}\ast\left(  \widetilde{p}_{r}^{\prime
}-\widetilde{p}_{r}^{\prime\prime}\right)  dre^{\left(  t-s\right)  L^{\ast}%
}\widetilde{\phi}\right\rangle \right\vert \\
&  \leq\gamma\left\Vert \frac{1}{1+\left\vert \mathbf{v}\right\vert
}e^{\left(  t-s\right)  L^{\ast}}\widetilde{\phi}\right\Vert _{\infty}h\left(
\int_{0}^{s}\left\Vert K_{2}\ast\widetilde{p}_{r}^{\prime}\right\Vert
_{\infty}dr\right)  d\left(  p_{s}^{\prime},p_{s}^{\prime\prime}\right) \\
&  +\gamma\left\Vert h^{\prime}\right\Vert _{\infty}\left\langle p_{s}%
^{\prime\prime},1+\left\vert \mathbf{v}\right\vert \right\rangle \left\Vert
\frac{1}{1+\left\vert \mathbf{v}\right\vert }e^{\left(  t-s\right)  L^{\ast}%
}\widetilde{\phi}\right\Vert _{\infty}\int_{0}^{s}\left\Vert K_{2}\ast\left(
\widetilde{p}_{r}^{\prime}-\widetilde{p}_{r}^{\prime\prime}\right)
\right\Vert _{\infty}dr.
\end{align*}
Moreover, we estimate%
\[
I_{s,t}^{2}\leq\left\Vert F\right\Vert _{\infty}\left\Vert \frac
{1}{1+\left\vert \mathbf{v}\right\vert }\nabla_{v}e^{\left(  t-s\right)
L^{\ast}}\widetilde{\phi}\right\Vert _{\infty}d\left(  p_{s}^{\prime}%
,p_{s}^{\prime\prime}\right)  ;
\]%
\[
I_{s,t}^{3}\leq\left\langle p_{s}^{\prime\prime},1+\left\vert \mathbf{v}%
\right\vert \right\rangle \left\Vert \frac{1}{1+\left\vert \mathbf{v}%
\right\vert }\nabla_{v}e^{\left(  t-s\right)  L^{\ast}}\widetilde{\phi
}\right\Vert _{\infty}\left\Vert \nabla F\right\Vert _{\infty}\left\Vert
\nabla C_{s}^{\prime}-\nabla C_{s}^{\prime\prime}\right\Vert _{\infty}.
\]
Now we use Lemma \ref{lemma Vlasov semigroup} below, the finiteness of
$\left\Vert G_{\mathbf{v}_{0}}\left(  1+\left\vert \mathbf{v}\right\vert
\right)  \right\Vert _{L^{1}},\left\Vert \alpha\right\Vert _{\infty
},\left\Vert \alpha^{\prime}\right\Vert _{\infty},$ $\left\Vert \beta
\right\Vert _{\infty},\left\Vert \beta^{\prime}\right\Vert _{\infty
},\left\Vert K_{2}\right\Vert _{\infty},\left\Vert f\right\Vert _{\infty
},\left\Vert f^{\prime}\right\Vert _{\infty}$, the assumption $\left\Vert
\phi\right\Vert _{\infty}\leq1$ and the properties $p^{\prime},p^{\prime
\prime}\in C\left(  \left[  0,T\right]  ;\mathcal{M}_{1}\left(  \mathbb{R}%
^{2d}\right)  \right)  $  \newline (which implies $\pi_{1}p^{\prime},\pi_{1}%
p^{\prime\prime}\in C\left(  \left[  0,T\right]  ;\mathcal{M}_{1}\left(
\mathbb{R}^{d}\right)  \right)  $) and $\left\Vert \nabla C^{\prime
}\right\Vert _{\infty}<\infty$, to get%
\[
I_{s,t}^{1,1}\leq c_{1}d\left(  p_{s}^{\prime},p_{s}^{\prime\prime}\right)
+c_{2}\left\Vert C_{s}^{\prime}-C_{s}^{\prime\prime}\right\Vert _{\infty}%
\]%
\[
I_{s,t}^{1,2}\leq c_{3}\int_{0}^{s}d\left(  p_{r}^{\prime},p_{r}^{\prime
\prime}\right)  dr+c_{4}\left\Vert C_{s}^{\prime}-C_{s}^{\prime\prime
}\right\Vert _{\infty}%
\]%
\[
I_{s,t}^{1,3}\leq c_{5}d\left(  p_{s}^{\prime},p_{s}^{\prime\prime}\right)
+c_{6}\int_{0}^{s}d\left(  p_{r}^{\prime},p_{r}^{\prime\prime}\right)  dr
\]
(we have used the fact that $\left\vert \left(  K_{2}\ast\left(
\widetilde{p}_{r}^{\prime}-\widetilde{p}_{r}^{\prime\prime}\right)
\right) \left(  \mathbf{x}\right)  \right\vert $ is bounded above
by \newline
$\left\Vert K_{2}\right\Vert _{\infty}d\left(  p_{r}^{\prime},p_{r}%
^{\prime\prime}\right)  $ thanks to the presence of the factor $\left(
1+\left\vert \mathbf{v}\right\vert \right)  $ in the definition of the
distance)
\[
I_{s,t}^{2}\leq\frac{c_{7}}{\left\vert t-s\right\vert ^{1/2}}d\left(
p_{s}^{\prime},p_{s}^{\prime\prime}\right)
\]%
\[
I_{s,t}^{3}\leq\frac{c_{8}}{\left\vert t-s\right\vert ^{1/2}}\left\Vert \nabla
C_{s}^{\prime}-\nabla C_{s}^{\prime\prime}\right\Vert _{\infty}.
\]
It follows%
\begin{align*}
&  \left\vert \left\langle p_{t}^{\prime}-p_{t}^{\prime\prime},\phi
\right\rangle \right\vert \leq\int_{0}^{t}\left(  c_{1}+c_{5}+\frac{c_{7}%
}{\left\vert t-s\right\vert ^{1/2}}\right)  d\left(  p_{s}^{\prime}%
,p_{s}^{\prime\prime}\right)  ds\\
&  +\int_{0}^{t}\left(  \left(  c_{2}+c_{4}\right)  \left\Vert C_{s}^{\prime
}-C_{s}^{\prime\prime}\right\Vert _{\infty}+\frac{c_{8}}{\left\vert
t-s\right\vert ^{1/2}}\left\Vert \nabla C_{s}^{\prime}-\nabla C_{s}%
^{\prime\prime}\right\Vert _{\infty}\right)  ds\\
&  \quad+\int_{0}^{t}\left(  \left(  c_{3}+c_{6}\right)  \int_{0}^{s}d\left(
p_{r}^{\prime},p_{r}^{\prime\prime}\right)  dr\right)  ds.
\end{align*}

At the same time, from equation (\ref{PDE for C}), we deduce%
\[
\partial_{t}\left(  C_{t}^{\prime}-C_{t}^{\prime\prime}\right)  =d_{1}%
\Delta\left(  C_{t}^{\prime}-C_{t}^{\prime\prime}\right)  -\left(  \eta
_{t}\left(  p_{\cdot}^{\prime}\right)  -\eta_{t}\left(  p_{\cdot}%
^{\prime\prime}\right)  \right)  C_{t}^{\prime}-\eta_{t}\left(  p_{\cdot
}^{\prime\prime}\right)  \left(  C_{t}^{\prime}-C_{t}^{\prime\prime}\right)
\]
hence (the initial conditions being the same)%
\begin{align*}
\left(  1-A\right)  ^{1/2}C_{N}\left(  t\right)   &  =-\int_{0}^{t}\left(
1-A\right)  ^{1/2}e^{\left(  t-s\right)  A}\left(  \eta_{s}\left(  p_{\cdot
}^{\prime}\right)  -\eta_{s}\left(  p_{\cdot}^{\prime\prime}\right)  \right)
C_{s}^{\prime}ds\\
&  -\int_{0}^{t}\left(  1-A\right)  ^{1/2}e^{\left(  t-s\right)  A}\eta
_{s}\left(  p_{\cdot}^{\prime\prime}\right)  \left(  C_{s}^{\prime}%
-C_{s}^{\prime\prime}\right)  ds
\end{align*}
which gives us%
\begin{align*}
\left\Vert \nabla C_{t}^{\prime}-\nabla C_{t}^{\prime\prime}\right\Vert
_{\infty}  &  \leq\int_{0}^{t}\frac{c_{10}}{\left\vert t-s\right\vert ^{1/2}%
}\left\Vert \eta_{s}\left(  p_{\cdot}^{\prime}\right)  -\eta_{s}\left(
p_{\cdot}^{\prime\prime}\right)  \right\Vert _{\infty}ds\\
&  +\int_{0}^{t}\frac{c_{10}}{\left\vert t-s\right\vert ^{1/2}}\left\Vert
\eta_{s}\left(  p_{\cdot}^{\prime\prime}\right)  \right\Vert _{\infty
}\left\Vert C_{s}^{\prime}-C_{s}^{\prime\prime}\right\Vert _{\infty}ds
\end{align*}
and a similar easier estimate for $\left\Vert C_{t}^{\prime}-C_{t}%
^{\prime\prime}\right\Vert _{\infty}$. Since $\left\Vert K_{1}\right\Vert
_{\infty}<\infty$, we have $\left\Vert \eta_{\cdot}\left(  p_{\cdot}%
^{\prime\prime}\right)  \right\Vert _{\infty}<\infty$ and
\[
\left\Vert \eta_{s}\left(  p_{\cdot}^{\prime}\right)  -\eta_{s}\left(
p_{\cdot}^{\prime\prime}\right)  \right\Vert _{\infty}\leq\left\Vert
K_{1}\right\Vert _{\infty}\int_{0}^{s}d\left(  p_{r}^{\prime},p_{r}%
^{\prime\prime}\right)  dr.
\]

Putting all together we find an integral inequality for the quantity%
\[
\sup_{r\in\left[  0,t\right]  }d\left(  p_{r}^{\prime},p_{r}^{\prime\prime
}\right)  +\left\Vert C_{t}^{\prime}-C_{t}^{\prime\prime}\right\Vert _{\infty
}+\left\Vert \nabla C_{t}^{\prime}-\nabla C_{t}^{\prime\prime}\right\Vert
_{\infty}%
\]
to which a generalized form of Gronwall inequality can be applied. It implies
$d\left(  p_{t}^{\prime},p_{t}^{\prime\prime}\right)  +\left\Vert
C_{t}^{\prime}-C_{t}^{\prime\prime}\right\Vert _{\infty}=0$, namely uniqueness.

\begin{lemma}
\label{lemma Vlasov semigroup}On a finite interval $\left[  0,T\right]  $,
there is a constant $C>0$ such that
\begin{align*}
\left\vert \left(  e^{tL^{\ast}}\left(  1+\left\vert \mathbf{v}\right\vert
\right)  \phi\right)  \left(  \mathbf{x},\mathbf{v}\right)  \right\vert  &
\leq C\left\Vert \phi\right\Vert _{\infty}\left(  1+\left\vert \mathbf{v}%
\right\vert \right) \\
\left\vert \nabla_{v}\left(  e^{tL^{\ast}}\left(  1+\left\vert \mathbf{v}%
\right\vert \right)  \phi\right)  \left(  \mathbf{x},\mathbf{v}\right)
\right\vert  &  \leq\frac{C}{\sqrt{t}}\left\Vert \phi\right\Vert _{\infty
}\left(  1+\left\vert \mathbf{v}\right\vert \right)  +C\left\Vert
\phi\right\Vert _{\infty}.
\end{align*}

\end{lemma}

Proof. Step 1. The generator $L^{\ast}=\frac{\sigma^{2}}{2}\Delta_{v}%
\phi+v\cdot\nabla_{x}\phi+k_{1}v\cdot\nabla_{v}\phi$ is associated with the
system%
\begin{align*}
dx_{t}  &  =v_{t}dt\\
dv_{t}  &  =k_{1}v_{t}dt+\sigma dB_{t}%
\end{align*}
where $B_{t}$ is an auxiliary Brownian motion on $R^{d}$. Set $z_{t}%
=e^{-k_{1}t}v_{t}$; we have%
\[
dz_{t}=e^{-k_{1}t}\sigma dB_{t}.
\]
Using this trick, the computations in the case $k_{1}\neq0$ are very similar
to those of the case $k_{1}=0$, just more cumbersome. We thus set $k_{1}=0$
for simplicity of notations; and we take $\sigma=1$ for the same reason. In
this case the solution of the system, called $\left(  \mathbf{x}%
,\mathbf{v}\right)  $ the initial condition,
\begin{align*}
v_{t}  &  =\mathbf{v}+B_{t}\\
x_{t}  &  =\mathbf{x}+\mathbf{v}t+\int_{0}^{t}B_{s}ds.
\end{align*}

We use the probabilistic formula
\[
\left(  e^{tL^{\ast}}\phi\right)  \left(  \mathbf{x},\mathbf{v}\right)
=E\left[  \phi\left(  \mathbf{x}+\mathbf{v}t+\int_{0}^{t}B_{s}ds,\mathbf{v}%
+B_{t}\right)  \right]  .
\]
One can prove%
\[
\left(  \nabla_{v}e^{tL^{\ast}}\phi\right)  \left(  \mathbf{x},\mathbf{v}%
\right)  =\frac{6}{t}E\left[  \left(  \frac{1}{t}\int_{0}^{t}B_{s}ds-\frac
{1}{3}B_{t}\right)  \phi\left(  \mathbf{x}+\mathbf{v}t+\int_{0}^{t}%
B_{s}ds,\mathbf{v}+B_{t}\right)  \right]  .
\]

Step 2. We thus have
\[
\left(  e^{tL^{\ast}}\left(  1+\left\vert \mathbf{v}\right\vert \right)
\phi\right)  \left(  \mathbf{x},\mathbf{v}\right)  =E\left[  \left(
1+\left\vert \mathbf{v}+B_{t}\right\vert \right)  \phi\left(  \mathbf{x}%
+\mathbf{v}t+\int_{0}^{t}B_{s}ds,\mathbf{v}+B_{t}\right)  \right]
\]%
\begin{align*}
\left(  \nabla_{v}e^{tL^{\ast}}\left(  1+\left\vert \mathbf{v}\right\vert
\right)  \phi\right)  \left(  \mathbf{x},\mathbf{v}\right)   &  =\frac{6}%
{t}E\left[  \left(  \frac{1}{t}\int_{0}^{t}B_{s}ds-\frac{1}{3}B_{t}\right)
\left(  1+\left\vert \mathbf{v}+B_{t}\right\vert \right)  \right.  \times\\
&  \qquad\qquad\times\left.  \phi\left(  \mathbf{x}+\mathbf{v}t+\int_{0}%
^{t}B_{s}ds,\mathbf{v}+B_{t}\right)  \right]  .
\end{align*}
Hence%
\[
\left\vert \left(  e^{tL^{\ast}}\left(  1+\left\vert \mathbf{v}\right\vert
\right)  \phi\right)  \left(  \mathbf{x},\mathbf{v}\right)  \right\vert
\leq\left\Vert \phi\right\Vert _{\infty}\left(  1+\sqrt{t}+\left\vert
\mathbf{v}\right\vert \right)
\]
which gives us the first bound;\ and
\begin{align*}
\left\vert \left(  \nabla_{v}e^{tL^{\ast}}\left(  1+\left\vert \mathbf{v}%
\right\vert \right)  \phi\right)  \left(  \mathbf{x},\mathbf{v}\right)
\right\vert  &  \leq\left\Vert \phi\right\Vert _{\infty}\frac{6}{t}E\left[
\frac{1}{t}\left(  \int_{0}^{t}\left\vert B_{s}\right\vert ds\right)  \left(
1+\left\vert \mathbf{v}\right\vert +\left\vert B_{t}\right\vert \right)
\right] \\
&  +\left\Vert \phi\right\Vert _{\infty}\frac{2}{t}E\left[  \left\vert
B_{t}\right\vert \left(  1+\left\vert \mathbf{v}\right\vert +\left\vert
B_{t}\right\vert \right)  \right]
\end{align*}%
\[
\leq\frac{C}{\sqrt{t}}\left\Vert \phi\right\Vert _{\infty}\left(  1+\left\vert
\mathbf{v}\right\vert \right)  +C\left\Vert \phi\right\Vert _{\infty}%
\]
which gives the second bound.

\appendix

\section{Strict positivity of the solution of the mean field PDE system}

\label{B}

Let $(p,C)$ be the solution on $\left[  0,T\right]  $ of the PDE system, with
$p_{0}$ a non negative measure. For every $t\in\left[  0,T\right]  $, $p_{t}$
is also a non negative measure. Let us denote by
\[
M_{t}:=\int_{\mathbb{R}^{d}\times\mathbb{R}^{d}}p_{t}\left(  d\mathbf{x}%
,d\mathbf{v}\right)
\]
its total mass. Moreover, let us denote by
\[
M_{t}^{N}:=\int_{\mathbb{R}^{d}\times\mathbb{R}^{d}}Q_{N}\left(  t\right)
\left(  d\mathbf{x},d\mathbf{v}\right)
\]
the total empirical mass.

The following theorem excludes extinction of the tip cells, in the PDE limit,
during any finite time interval $[0,T];$ as a consequence, for large $N,$ the
same holds for the random empirical measure of tips.

\begin{theorem}
If $M_{0}>0$ then, for any choice of $T>0,$
\[
M_{t}>0\text{ for every }t\in\left[  0,T\right]  .
\]
More precisely, there is a constant $C_{p_{0},T}>0$, depending on $p_{0}$ and
$T$ such that $M_{t}\geq C_{p_{0},T}$ for all $t\in\left[  0,T\right]  $. Due
to the weak convergence result for the empirical measures, we also have%
\[
\lim_{N\rightarrow\infty}P\left(  M_{t}^{N}\geq C_{p_{0},T}/2\text{ for all
}t\in\left[  0,T\right]  \right)  =1.
\]

\end{theorem}

\begin{proof}
In the weak formulation (21) of the equation for $p_{t}$, let us take the test
function $\phi\left(  \mathbf{x},\mathbf{v}\right)  $ identically equal to 1
(more precisely, one has to take the limit of test functions converging to 1;
we omit the details). We get, with $g_{0}:=\int_{\mathbb{R}^{d}}G_{v_{0}%
}\left(  \mathbf{v}\right)  d\mathbf{v}$,%
\begin{align*}
M_{t} &  =M_{0}+g_{0}\int_{0}^{t}\int_{\mathbb{R}^{d}}\alpha\left(
C_{s}\left(  \mathbf{x}\right)  \right)  \left(  \pi_{1}p_{s}\right)  \left(
d\mathbf{x}\right)  ds\\
&  +g_{0}\int_{0}^{t}\int_{\mathbb{R}^{d}}\beta\left(  C_{s}\left(
\mathbf{x}\right)  \right)  \int_{0}^{s}\widetilde{p}_{r}\left(
d\mathbf{x}\right)  drds\\
&  -\gamma\int_{0}^{t}\int_{\mathbb{R}^{d}\times\mathbb{R}^{d}}h\left(
\int_{0}^{s}\left(  K_{2}\ast\widetilde{p}_{r}\right)  \left(  \mathbf{x}%
\right)  dr\right)  p_{s}\left(  d\mathbf{x},d\mathbf{v}\right)  ds.
\end{align*}
Since $M_{0}>0$ and the function $M_{t}$ is continuous ($p_{t}$ is weakly
continuous), there is an open interval $\left(  0,\tau\right)  $ where
$M_{t}>0$, such that either $\tau=+\infty$, or $\tau<\infty$ and $M_{\tau}=0$.
We have to exclude the second case. For $\,t\in\left(  0,\tau\right)  $ we
have%
\[
-\log M_{t}=-\log M_{0}-\int_{0}^{t}\frac{1}{M_{s}}\frac{d}{ds}M_{s}ds
\]
hence, from the previous identity%
\begin{align*}
-\log M_{t} &  =-\log M_{0}-\int_{0}^{t}\frac{1}{M_{s}}g_{0}\int%
_{\mathbb{R}^{d}}\alpha\left(  C_{s}\left(  \mathbf{x}\right)  \right)
\left(  \pi_{1}p_{s}\right)  \left(  d\mathbf{x}\right)  ds\\
&  -\int_{0}^{t}\frac{1}{M_{s}}g_{0}\int_{\mathbb{R}^{d}}\beta\left(
C_{s}\left(  \mathbf{x}\right)  \right)  \int_{0}^{s}\widetilde{p}_{r}\left(
d\mathbf{x}\right)  drds\\
&  +\int_{0}^{t}\frac{1}{M_{s}}\gamma\int_{\mathbb{R}^{d}\times\mathbb{R}^{d}%
}h\left(  \int_{0}^{s}\left(  K_{2}\ast\widetilde{p}_{r}\right)  \left(
\mathbf{x}\right)  dr\right)  p_{s}\left(  d\mathbf{x},d\mathbf{v}\right)  ds.
\end{align*}
Since $g_{0}\geq0$, $\alpha\left(  \cdot\right)  \geq0$, $\beta\left(
\cdot\right)  \geq0$, $\pi_{1}p_{s}$ and $\widetilde{p}_{r}$ are non-negative
measures, the first two integral terms are positive, hence we have%
\[
-\log M_{t}\leq-\log M_{0}+\int_{0}^{t}\frac{1}{M_{s}}\gamma\int%
_{\mathbb{R}^{d}\times\mathbb{R}^{d}}h\left(  \int_{0}^{s}\left(  K_{2}%
\ast\widetilde{p}_{r}\right)  \left(  \mathbf{x}\right)  dr\right)
p_{s}\left(  d\mathbf{x},d\mathbf{v}\right)  ds.
\]
Since $h$ is bounded we have
\begin{align*}
-\log M_{t} &  \leq-\log M_{0}+\gamma\int_{0}^{t}\frac{1}{M_{s}}%
\int_{\mathbb{R}^{d}\times\mathbb{R}^{d}}p_{s}\left(  d\mathbf{x}%
,d\mathbf{v}\right)  ds\\
&  =-\log M_{0}+\gamma\int_{0}^{t}1ds\leq-\log M_{0}+\gamma T.
\end{align*}
It follows that%
\[
M_{t}\geq C_{p_{0},T}:=\exp\left(  \log M_{0}-\gamma T\right)  >0.
\]
The last claim of the theorem, on $M_{t}^{N}$, is direct consequence of the
convergence in probability of $Q_{N}\left(  t\right)  $ to $p_{t}$, in
$L^{\infty}\left(  0,T;\mathcal{M}_{1}\left(  \mathbb{R}^{d}\times
\mathbb{R}^{d}\right)  \right)  $.
\end{proof}

\end{document}